\documentclass[12pt]{amsart}

\usepackage{amsmath,amssymb,color}
\usepackage{amsthm}
\usepackage{hyperref}
\usepackage[margin=1.0in]{geometry}
\usepackage{cleveref}
\usepackage{tikz-cd}

\numberwithin{equation}{section}

\newtheorem{prop}{Proposition}
\newtheorem{lemma}[prop]{Lemma}

\newtheorem{thm}[prop]{Theorem}
\newtheorem{cor}[prop]{Corollary}

\numberwithin{prop}{section}

\theoremstyle{definition}
\newtheorem{defn}[prop]{Definition}

\newtheorem{rmk}[prop]{Remark}

\newcommand{\del}{\partial}
\newcommand{\delb}{\bar{\partial}}
\newcommand{\brs}[1]{\left| #1 \right|}

\renewcommand{\gg}{\gamma}
\newcommand{\gD}{\Delta}

\newcommand{\gl}{\lambda}

\newcommand{\gw}{\omega}
\newcommand{\ga}{\alpha}
\newcommand{\gb}{\beta}

\newcommand{\N}{\nabla}

\renewcommand{\bar}[1]{\overline{#1}}

\renewcommand{\i}{\sqrt{-1}}

\newcommand{\IP}[1]{\left<#1\right>}

\newcommand{\Scal}{R}

\DeclareMathOperator{\SU}{SU}

\DeclareMathOperator{\Rc}{Rc}

\DeclareMathOperator{\tr}{tr}

\DeclareMathOperator{\spn}{span}

\begin{document}

\title[Rigidity results for non-K\"ahler Calabi-Yau geometries on threefolds]{Rigidity results for non-K\"ahler Calabi-Yau geometries on threefolds}

\author{Vestislav Apostolov}
\address{V.\,Apostolov\\ D{\'e}partement de Math{\'e}matiques\\ UQAM \\
 and \\ Institute of Mathematics and Informatics\\ Bulgarian Academy of Sciences}
\email{\href{mailto:apostolov.vestislav@uqam.ca}{apostolov.vestislav@uqam.ca}}

\author{Giuseppe Barbaro}
\address{G.\,Barbaro\\ Department of Mathematics, Aarhus University, Ny Munkegade 118, 8000 Aarhus C,
Denmark}
\email{\href{mailto:g.barbaro@math.au.dk}{g.barbaro@math.au.dk}}

\author{Kuan-Hui Lee}
\address{K-H. Lee, Rowland Hall\\
          University of California, Irvine\\
          Irvine, CA 92617}
 \email{\href{mailto:kuanhuil@uci.edu}{kuanhuil@uci.edu}}

 \author{Jeffrey Streets}
 \address{J. Streets Rowland Hall\\
          University of California, Irvine\\
          Irvine, CA 92617}
 \email{\href{mailto:jstreets@uci.edu}{jstreets@uci.edu}}
\thanks{V.A. was supported in part by an NSERC Discovery Grant. G.B. is a member of GNSAGA of INdAM and has been supported by Sapere Aude: DFF-Starting Grant “Conformal geometry: metrics and cohomology”.  K-H.L. and J.S. are both supported by the NSF via DMS-2203536}

\date{April 23, 2025}

\begin{abstract}  We derive a canonical symmetry reduction associated to a compact non-K\"ahler Bismut-Hermitian-Einstein manifold.  In real dimension $6$, the transverse geometry is conformally K\"ahler, and we give a complete description in terms of a single scalar PDE for the underlying K\"ahler structure.  In the case when the soliton potential is constant, we show that that the Bott-Chern number $h^{1,1}_{BC} \geq 2$, and that equality holds if and only if the metric is Bismut-flat, and hence a quotient of either $\SU(2) \times \mathbb R \times \mathbb C$ or $\SU(2) \times \SU(2)$.
\end{abstract}

\maketitle

\section{Introduction}

Non-K\"ahler Calabi-Yau metrics are a subject of intense recent interest (e.g. \cite{fernandez2014non, finosurvey, JFS, garciafern2018canonical,ivanov2010heterotic,phong2019geometric,picard2024strominger,st-geom, tosatti2015non}), relevant to mathematical physics and uniformization problems in complex geometry.  Given $(M^{2n}, g, J)$ a Hermitian manifold, we say that the metric is pluriclosed if $\i \del \delb \gw = 0$.  There is a unique Hermitian connection with skew-symmetric torsion, which we refer to as the Bismut connection \cite{Bismut} (cf. also \cite{StromingerSST}).  This connection produces a representative of the first Chern class, the Bismut-Ricci form.  Pluriclosed metrics with vanishing Bismut-Ricci form are called Bismut-Hermitian-Einstein (BHE) (cf. \cite{JFS, GRFBook}), and if $\gw$ is K\"ahler they are Ricci-flat, i.e. are Calabi-Yau.  Such metrics are natural candidates for `canonical metrics' in complex non-K\"ahler geometry, arising for instance as fixed points of the pluriclosed flow \cite{PCF}.  Moreover the associated metric $g$ and three-form $H = -d^c \gw$ satisfy a supergravity equation arising from the string effective action \cite{Ivanovstring,Polchinski,PCFReg}.

Yau's theorem \cite{YauCC} shows the diverse array of K\"ahler Calabi-Yau manifolds in dimension three, with a large (conjecturally finite) number of diffeotypes, each endowed with a nontrivial moduli space of complex structures and K\"ahler Calabi-Yau metrics.  Naively, one may expect that upon allowing for non-K\"ahler geometries, we should thus expect a vast array of solutions.  There are hints that this is not the case: a remarkable result of Gauduchon-Ivanov states that for complex surfaces, the only BHE manifolds are Bismut-flat, and in fact quotients of the standard Bismut-flat Hopf surface \cite{GauduchonIvanov}.  Here we investigate these structures in higher dimensions.  First, we know that these metrics are automatically generalized Ricci solitons with soliton potential $f$ (\cite{JFS}, cf. Proposition \ref{p:BHEprop} below).  Furthermore, building on \cite{JFS} we show that there is a natural $\mathfrak t^2$-symmetry generated by the Bismut-parallel vector fields
\begin{align*}
    {V = (\theta^{\sharp} - \N f), \quad JV = J \left( \theta^{\sharp} - \N f \right)}.
\end{align*}
Using this symmetry, we show in Proposition \ref{p:EM} that the transverse geometry satisfies a kind of twisted Einstein-Maxwell equation.  In real dimension six we obtain a much more precise structure:

\begin{thm} \label{t:mainthm} Fix $(M^{6}, \gw, J)$ a compact Bismut-Hermitian-Einstein manifold which is not a K\"ahler Calabi-Yau manifold.  Then
\begin{enumerate}
    \item $(M, \gw, J)$ admits two commuting Killing vector fields of constant norm,  giving rise to a regular rank $2$ Riemannian foliation which admits an invariant transversal K\"ahler structure $(\omega_K^T, J^T)$ of everywhere positive transversal scalar curvature $\Scal(\omega_K^T)$, satisfying the equation
    \[ \tfrac{1}{2}\left((dd^c \Scal(\omega_K^T)) \wedge \omega_K^T\right) = \alpha^T \wedge \alpha^T + \rho_K^T \wedge \rho_K^T, \]
    where $\rho_K^T$ is the transversal Ricci form of $\omega_K^T$ and $\alpha^T$ is a transversal closed primitive $(1,1)$-form.  Conversely, any K\"ahler surface $(M^T,\omega_K^T, J^T, \alpha^T)$ satisfying the above conditions gives rise to a (locally defined) non-K\"ahler Bismut-Hermitian-Einstein manifold. If, furthermore, $M^T$  is smooth and compact,  and the vector subspace of  $H^2(M^T, \mathbb{R})$  spanned by  $[\alpha^T]$ and $ c_1(M^K, J^K)$ admits  a basis in $H^2(M, \mathbb{Z})$, then $(M^T, J^T, \omega^T_K, \alpha^T)$ gives rise to a compact non-K\"ahler Bismut-Hermitian-Einstein manifold.
    \item The soliton potential of $(M, \gw, J)$ is constant if and only if $\gw$ is Gauduchon, if and only if the transversal K\"ahler metric $\omega_K^T$ has constant scalar curvature. 
    \item The dimension $h^{1,1}_{BC}(M)= {\rm dim}_{\mathbb{C}}H^{1,1}_{\rm BC}(M, \mathbb{C})$ of the Bott-Chern cohomology group of $(M, J)$ is greater than or equal to $2$.
\end{enumerate}
\end{thm}

As a corollary of this structural theorem, we extend the strong rigidity result of Gauduchon-Ivanov to 3-folds, assuming that the Bott-Chern cohomolgy group $H^{1,1}_{\rm BC}(M, \mathbb{C})$ is 2-dimensional, and that the soliton potential is constant.  We also classify solutions under some other natural geometric constraints.

\begin{cor} \label{c:rigiditycor} Fix $(M^{6}, \gw, J)$ a compact Bismut-Hermitian-Einstein manifold which is not a K\"ahler Calabi-Yau manifold.   Then
\begin{enumerate}
    \item $(M, \gw, J)$ has constant soliton potential and $h^{1,1}_{\rm BC}(M)=2$ if and only if $(M, \omega, J)$ is isometric to a quotient of either $\SU(2) \times \mathbb R \times \mathbb C$ or $\SU(2) \times \SU(2)$, endowed with a bi-invariant metric and compatible left-invariant complex structure.
    \item $(M, \omega, J)$  has harmonic Lee form  if and only if it is isometric to a quotient of  $\SU(2) \times \mathbb R \times \mathbb C$, endowed with a bi-invariant metric and compatible left-invariant complex structure.
    \item If the transversal K\"ahler metric $\omega_K$ in Theorem~\ref{t:mainthm} is K\"ahler-Einstein, then $(M, \omega, J)$ is isometric to a quotient of $\SU(2) \times \SU(2)$, endowed with a bi-invariant metric and compatible left-invariant complex structure.
\end{enumerate}
\end{cor}

\begin{rmk} \ 
\begin{enumerate}
 \item In the case when the foliation $\mathcal V$ on $(M, J)$ is regular, i.e. the leaf space of $\mathcal V$ is a smooth complex surface $S=(M^T, J^T)$, Theorem~\ref{t:mainthm}-(2) yields that $S$ admits a K\"ahler metric of positive scalar curvature and that the first Chern number satisfies $c_1(S)^2\geq 0$. Using  Wu's formula  $c_1^2(S)= 2\chi(S) + 3\sigma(S)$ and  the fact that $b_+(S)=1$ for a K\"ahler surface admitting a K\"ahler metric with positive scalar curvature, we obtain the inequality  $b_2(S) \leq 10 - 4b_1(S)$. By Prop.~\ref{p:full-Betti} below,  at least in the cases when the soliton potential is constant or $b_1(M)=0$, we also obtain the inequality $h^{1,1}_{\rm BC}(M) \le 10$.

\item In \cite{grantcharov2008calabi} the authors construct both Hermitian metrics with vanishing Bismut-Ricci form (called there CYT metrics), and also pluriclosed metrics, on the manifolds $M_k=(k-1) (S^2 \times S^4) \# k (S^3 \times S^3), \, k\geq 1$.  The authors pose the question of whether a pluriclosed CYT metric (here BHE) exists on these manifolds. Their construction provides a regular holomorphic foliation $\mathcal{V}_k$ with regular leaf space a compact complex surface $S_k$ such that $b_2(S_k)=k+1$. Our results here show that $k \le 9$, should $M_k$ admit a BHE metric compatible with $\mathcal{V}_k$. We further obtain a classification of all  Gauduchon BHE structures in the case $k=1$.

    \item A related rigidity result with further hypotheses on the curvature recently appeared in (\cite{ivanov2023riemannian} Theorem 1.6)
    \item By \cite{ye2023bismut}, for $M$ compact the vanishing of the Bismut-Ricci form is implied by the vanishing of its $(1,1)$-component.
    \item It follows that all non-K\"ahler examples in Corollary \ref{c:rigiditycor} are covered by \emph{generalized K\"ahler manifolds}, by endowing the bi-invariant metric on the universal cover further with a right-invariant complex structure (cf. \cite{GualtieriThesis} Example 6.39).  However even in dimension two there are quotients of the standard flat Hopf surface which admit no compatible generalized K\"ahler structure \cite{ASU3}.
    \item In \cite{ASU3,Streetssolitons,SU1} a classification of the closely related generalized K\"ahler-Ricci solitons (GKRS) is given in complex dimension $n=2$.  A further important problem is to classify GKRS in dimension $n=3$, or pluriclosed solitons in dimensions $n=2,3$.
\end{enumerate}
\end{rmk}

To prove Theorem \ref{t:mainthm}, we first derive certain a priori symmetries which appear in all dimensions.  In \cite{JFS} the authors showed that every compact BHE manifold is a generalized Ricci soliton with potential function $f$, and moreover the vector fields $V = \theta^{\sharp} - \N f$ and $JV$ are real holomorphic Killing fields (cf. Proposition \ref{p:BHEprop} below).  These vector fields are nontrivial if the metric is not K\"ahler Calabi-Yau, in which case we obtain a reduction of the BHE equation in terms of the transverse geometry.  These reduced equations are closely linked to the Hull-Strominger system in mathematical physics \cite{Hull, StromingerSST}, and this symmetry reduction is related to recent works \cite{GFGMS, GMolina,Streetssolitons,SRYM2}.  We show in particular that the transverse Hermitian metric is always conformally balanced, and satisfies an anomaly cancellation equation.  Restricting now to the case $n = 3$, it follows that the transverse geometry is conformally K\"ahler, and furthermore satisfies a twisted Einstein-Maxwell equation (cf. \cite{apostolov2019conformally, hawking2023large,lebrun2015einstein, misner1973gravitation}) with respect to the associated principal curvatures.  Exploiting various special geometry of the four-dimensional transverse distribution, we derive items (1) and (2) of Theorem \ref{t:mainthm}.  Using the Bochner-Weitzenbock formulas together with this special geometry we derive various topological characteristics, including item (3) of Theorem \ref{t:mainthm}.  To establish Corollary \ref{c:rigiditycor}, we must obtain finer information from the associated anomaly cancellation equation.  In particular assuming the topological condition $h_{BC}^{1,1}(M) = 2$ it follows that $\ga^T$ and the primitive part of $\rho_K^T$ are proportional, and it follows that the eigenvalues of the transverse Ricci tensor are constant.  By adapting the argument of \cite{VTA} to our setting of a transverse K\"ahler geometry, it follows that the transverse Ricci tensor is parallel.  This in turn implies that the Bismut torsion on the total space is parallel, and then the proof concludes using the recent characterization of BHE manifolds with parallel Bismut torsion \cite{brienza2024cyt} (cf. also \cite{barbaro2024pluriclosed}) and the classification of Bismut-flat manifolds in low dimensions \cite{ZhengBismutflat}.  The other cases of Corollary \ref{c:rigiditycor} follow similarly.

\vskip 0.1in
\textbf{Acknowledgements:} We are very grateful to an anonymous referee for pointing out a crucial mistake in an earlier version of this paper.

\section{Dimension reduction} \label{s:dimred}

In this section we derive a toric reduction via a local ansatz for BHE manifolds, stemming from the a priori symmetry provided by the Lee vector field.  To begin we recall fundamental definitions and structural results for BHE metrics.  Fix $(M^{2n}, g, J)$ a Hermitian manifold, with K\"ahler form $\gw(X,Y) = g(JX, Y)$.  We say that the metric is pluriclosed if $\i\del\delb \gw = 0$.  Equivalently, introducing the real operator $d^c:= \sqrt{-1}(\bar{\partial} - \partial)$ we have $dd^c \omega=0$.  We set $H = - d^c \gw$ and then the pluriclosed condition is equivalent to $d H = 0$.  The associated Lee form is defined by $\theta = - d^* \gw \circ J$, and we say $(M^{2n},g,J)$ is balanced if $\theta=0$.  The associated Bismut connection $\nabla^B$ is defined by 
\begin{align*}
    \langle \nabla_X^{B}Y,Z \rangle=\langle \nabla_XY,Z \rangle- \tfrac{1}{2}d^c\omega(X,Y,Z) = \IP{\N_X Y, Z} + \tfrac{1}{2} H(X,Y,Z).
\end{align*}
and the Bismut Ricci form $\rho_B$ and Bismut Ricci tensor $\Rc^B$ are then defined by
\begin{align*}
    \rho_{B}(X,Y) := \tfrac{1}{2} \sum_{i=1}^{2n}\langle R^{B}(X,Y)Je_i,e_i\rangle,\qquad \Rc^B(X,Y) = \sum_{i=1}^{2n}\langle R^{B}(e_i,X) Y,e_i\rangle,
\end{align*}
for any orthonormal basis $\{e_i\}$, where $R^B$ is the Bismut Riemannian curvature.  

\subsection{Curvature identities}

We record some elementary computations which are useful in later sections (cf. \cite{GRFBook,Ivanovstring}).


\begin{lemma}
    Let $(M^{2n},g,J)$ be a Hermitian manifold. Then,
    \begin{align} 
        \theta =\tfrac{1}{2}\tr_{\omega}d\omega\label{2.2}
    \end{align}
    where for a $p$-form  $\psi$, we let $\tr_{\omega}\psi : = \sum_{i,j=1}^n\omega(e_i, e_j)\psi(e_i, e_j, \cdot, \ldots, \cdot)$ with respect to any orthonormal basis $\{e_i\}$.
\end{lemma}
\begin{proof}

Fix a point $p$ and choose a $J$-adapted orthonormal basis $\{e_1,e_2,...,e_{2n}\}$ at $p$. From the definition of $\theta$ and using that $\nabla_{X} \omega = g (\nabla_{X} J)$ has type $(2,0)+(0,2)$, we compute
\begin{align*}
    \theta(X)=-(d^*\omega)(JX)= \sum_{m=1}^{2n}(\nabla_{e_{m}}\omega)(e_m,JX)=\sum_{m=1}^{2n}(\nabla_{e_m} \omega)(Je_m, X).
\end{align*}
Using
\[ d\omega(Y, Z, X) = (\nabla_Y \omega)(Z, X) + (\nabla_X \omega)(Y, Z) -(\nabla_Z \omega)(Y, X), \]
and again that $(\nabla_X\omega)$ is $(2,0)+(0,2)$,  we get 
\[ \sum_{m=1}^{2n} d\omega(e_m, Je_m, X)= 2\sum_{m=1}^{2n} (\nabla_{e_m} \omega)(Je_m, X),\]
and hence
\begin{align*}
\theta(X)=\tfrac{1}{2}\sum_{m=1}^{2n}(d\omega)(e_m, Je_m, X) = \tfrac{1}{2} \tr_{\omega} d\omega, \end{align*}
as claimed.
\end{proof}

\begin{lemma}[\cite{GRFBook}, Proposition 8.10]\label{Lrho}
    Let $(M^{2n},g,J)$ be a pluriclosed, Hermitian manifold and $H=-d^c\omega$. Then,
    \begin{align}
        \rho_B^{1,1}(\cdot,J\cdot)=\Rc-\tfrac{1}{4}H^2+\tfrac{1}{2}\mathcal{L}_{\theta^{\sharp}}g, \quad \rho_B^{2,0+0,2}(\cdot,J\cdot)=-\tfrac{1}{2}d^*H+\tfrac{1}{2}d \theta-\tfrac{1}{2}\iota_{\theta^{\sharp}}H,\label{2.1}
    \end{align}
    where $\Rc$ denotes the Ricci tensor of $g$, $\theta^{\sharp}:= g^{-1} \theta$ is the Lee vector field, and $H^2(X,Y) := \IP{i_X H, i_Y H}_g$.
\end{lemma}

\subsection{Structure of Bismut Hermitian-Einstein manifolds}

In this subsection, we derive the a priori symmetries canonically associated to a BHE manifold.  To begin we review some basic structural results about these metrics and the closely related steady pluriclosed solitons.  For more details see \cite{GRFBook, lee2024stability}.

\begin{defn}\label{soliton}
We say that a pluriclosed manifold $(M^{2n},g,J)$ is a
\begin{enumerate}
    \item \emph{Bismut-Hermitian-Einstein manifold} (BHE) if $\rho_B = 0$.
    \item \emph{steady pluriclosed soliton} if there exists a smooth function $f$ such that 
\begin{align*}
   0=\Rc-\tfrac{1}{4}H^2+\nabla^2 f, \quad 0=d_g^*H+i_{\nabla f}H.
\end{align*}
    \item \emph{generalized Einstein manifold} if 
\begin{align*}
  0=\Rc-\tfrac{1}{4}H^2, \quad 0=d_g^*H \qquad \longleftrightarrow \qquad \Rc^B = 0
\end{align*}
\end{enumerate}
\end{defn}

\begin{lemma}[\cite{SU2} Proposition 4.1] \label{LV}
Let $(M^{2n},g,H,J,f)$ be a steady pluriclosed soliton with $H=-d^c\omega$. Define the vector field    
\begin{align*}
   { V=(\theta^\sharp-\nabla f)}.
\end{align*}
$V$ satisfies that 
\begin{align*}
    L_V J= L_{JV} J = 0,\quad L_{JV} g =0,\quad L_V\omega ={2}\rho_B^{1,1}.
\end{align*}
\end{lemma}

\begin{prop}[\cite{lee2024stability} Proposition 2.16]\label{P1}
Let $(M^{2n},g,H,J,f)$ be a steady pluriclosed soliton. Then one has
$$ \nabla^B(\theta-df)= \rho_B\circ J .$$
\end{prop}

We now record the main basic structural result for BHE manifolds.  Note that item (5) was previously proved in \cite[Proposition 2.6]{JFS}, relying on a result in \cite{ivanov2013vanishing}.  We give a direct proof here relying on soliton identities.

\begin{prop} \label{p:BHEprop} Let $(M^{2n}, g, J)$ be a compact Bismut-Hermitian-Einstein manifold.  Then the following hold:
\begin{enumerate}
    \item $(g, J)$ is a steady pluriclosed soliton with a unique normalized potential $f$;
    \item  $\nabla^BV=\nabla^BJV=0$;
    \item $\brs{V}^2 = \brs{J V}^2$ is constant;
    \item $L_V g = 0$;
    \item $V$ vanishes if and only if $g$ is K\"ahler.
\end{enumerate}
\begin{proof}
By (\cite[Proposition 2.6]{JFS}), it follows that any BHE manifold is a steady pluriclosed soliton.  Roughly speaking, the BHE equation implies that the pair $(g, H)$ is a generalized Ricci soliton, and then the monotonicity of the Perelman-type functional along generalized Ricci flow implies that in fact it is a gradient soliton.  Given this, by Proposition \ref{P1}, $\nabla^B V=0$. Since $\nabla^Bg=\nabla^BJ=0$, we also get that $\nabla^B J V = 0$ and that the norm of these vector fields is constant.
Moreover, being Bismut-parallel, $V$ is also Killing.  To show (5), recall that the pluriclosed assumption implies that (\cite[Equation 3.24]{Ivanovstring})
\begin{align*}
    |\theta|^2+d^*\theta-\tfrac{1}{6}|H|^2=0.
\end{align*}
We next use the Perelman $\gl$-invariant for a generalized Ricci soliton.  Recall that 
\begin{align*}
    \lambda(g,H)=\inf_{\{f| \int_Me^{-f}dV_g=1\}}\Big\{\int_M (R-\frac{1}{12}|H|^2+|\nabla f|^2)e^{-f}dV_g\Big\}.
\end{align*}
Note that the minimizer $f$ is uniquely achieved (cf. \cite{GRFBook} Ch. 6) and satisfies
\begin{align*}
    \lambda(g,H)=R-\frac{1}{12}|H|^2+2\gD f-|\nabla f|^2,
\end{align*}
where we use the convention $\Delta =-d^*d -dd^*$ for the Laplacian $\Delta$.  Since $(M,g,H)$ is a steady pluriclosed soliton, 
using the identities of (\cite[Proposition 4.33] {GRFBook}) we get
\begin{align*}
    \int_M \tfrac{1}{6}\brs{H}^2 e^{-f}dV_g =&\ \lambda(g,H)\\
    =&\ \tfrac{1}{6}|H|^2+\Delta f-|\nabla f|^2
    \\
    =&\ |\theta|^2+d^*\theta+\Delta f-|\nabla f|^2.
\end{align*}
Since $V$ is Killing we take the trace of $L_V g = 0$ to obtain
\begin{align*}
    d^*\theta+\Delta f=0.
\end{align*}
Thus, if $V = 0$ then $\theta=df$ and hence $H=0$.
\end{proof}
\end{prop}

\subsection{Invariant geometry}

In this subsection, we use the a priori symmetries guaranteed by Lemma \ref{LV} and Proposition \ref{p:BHEprop} to express the geometry locally in terms of invariant structures on a principal bundle.  Some related computations have appeared in the literature \cite{GFGM, GFGMS, GMolina, Streetssolitons}.  We assume throughout that $V \neq 0$, otherwise we would fall in the K\"ahler case by (5) in Proposition \ref{p:BHEprop} (see also \cite[Proposition 2.6]{JFS}).  In this case, it follows from Proposition \ref{p:BHEprop} that the vector fields $V$ and $JV$
are nontrivial Killing vector fields of constant length.  From now on we assume that all given BHE manifolds are scaled so that
\begin{align*}
    \brs{V}  = \brs{JV} \equiv 1.
\end{align*}
We define a distribution
\begin{align*}
\mathcal {V} = \spn \{V, JV\}.
\end{align*}
As $V$ and $JV$ are holomorphic Killing fields, $\mathcal {V}$ is integrable. 
Thus $M$ is locally equivalent to an open set in an $\mathbb R^2$-principal bundle with vertical space $\mathcal V$.  More abstractly, let $\mathfrak{t}^2$ denote the Lie algebra of $\mathbb R^2$. We will think of $V$ and $JV$ as the fundamental vector fields of a basis of $\mathfrak{t}^2$ (which we will still denote by $\{V, JV\}$) and we can define a connection $1$-form $\mu$
\begin{align}\label{eq:eta}
        \mu(X)=\eta(X)V+J\eta(X)JV, \qquad \eta:= \theta-df, \, \, J\eta:= -(\theta - df)\circ J.
 \end{align}
It follows easily from Proposition \ref{p:BHEprop} that this is indeed a principal connection, i.e.
\begin{align*}
    \mu(V) = V, \quad \mu(JV) = JV, \quad \mu(JX) = J_{\mathfrak t^2} \mu(X), \quad L_V \mu = L_{JV} \mu = 0.
\end{align*}
Thus, we define the horizontal distribution $\mathcal{H}=\ker\mu$.  It follows from the construction and standard facts about principal connections that
\begin{align*}
     [\mathcal{V},\mathcal{V}]\subset \mathcal{V}, \quad  [\mathcal{V},\mathcal{H}]\subset \mathcal{H},\quad J\mathcal{V}=\mathcal{V}, \quad J\mathcal{H}=\mathcal{H}.
\end{align*}
Furthermore, the distributions $\mathcal V$ and $\mathcal H$ are $g$-orthogonal, and we may decompose the K\"ahler form into the vertical and transverse pieces 
\begin{align*}
    \omega^V(X,Y)=&\ (\eta\wedge J\eta)(X,Y),\\
    \omega^T(X,Y)=&\ \omega(X,Y)-(\eta\wedge J\eta)(X,Y).
\end{align*}
We let $F:=F_{V} V +F_{JV} JV$ denote the principal curvature associated to $\mu$, i.e.  
\begin{align*}
    F_{V} : = d\eta = d \theta, \qquad F_{J V} : = d(J\eta) = d J \theta - dJdf,
\end{align*}
are closed forms supported on $\mathcal H$.  Lastly observe that the torsion of the Bismut connection is
\begin{equation}\label{eq: torsion form}
    H = H^T + J F_V\wedge \eta + JF_{JV}\wedge J\eta,
\end{equation}
where $H^T$ is the horizontal component of the torsion $3$-form, which amounts to $H^T=-Jd\omega^T$.

Using this setup we derive some key properties of the torsion and principal curvature.

\begin{lemma}\label{L3}
Let $(M^{2n},g,J,H)$ be a BHE manifold. Then $F_{V}$ and $F_{J V}$  are of type $(1,1)$ and are ${\mathcal V}$-basic.
\end{lemma}
\begin{proof}
    First of all, it holds
    $$\iota_V F_V = \iota_V d\eta = \mathcal{L}_V \eta - d\iota_V\eta = 0 ,$$
    and,
    $$\mathcal{L}_V F_V = \iota_V dF_V + d\iota_VF_V = 0 .$$
    The other combinations are similar, hence $F_V,F_{JV}$ are basic.
    Now using the Koszul formula, we get that
    \begin{align*}
        2\left<\nabla_{e_i}V,e_j\right> =&\  F_V(e_i,e_j),\\
        2\left<\nabla_{e_i}JV,e_j\right> =&\  F_{JV}(e_i,e_j).
    \end{align*}
    Since the difference between the Bismut and Levi Civita connection is given by half of the torsion $H$, we obtain
    \begin{align*}
        2\left<\nabla^B_{e_i}V,e_j\right> =&\  F_V(e_i,e_j) - JF_V(e_i,e_j),\\
        2\left<\nabla^B_{e_i}JV,e_j\right> =&\ F_{JV}(e_i,e_j) - JF_{JV}(e_i,e_j).
    \end{align*}
    The statement follows since the vectors $V$ and $JV$ are $\nabla^B$-parallel.
\end{proof}

\begin{rmk}\label{r:riem-fol} $(M, \mathcal{V})$ is an instance of a  \emph{Hermitian foliation}. Indeed, the  tensor 
\[g_T = g - \eta\otimes \eta - J\eta\otimes J\eta\]
is $\mathcal V$-basic and defines a Riemannian structure on the horizontal distribution  $\mathcal{H} \cong TM/\mathcal{V}$, i.e. $(M, \mathcal{V})$ is a Riemannian foliation (see e.g. \cite{Molino}). Similarly, the restriction of the complex structure $J$ to $\mathcal{H}$ gives rise to a transversal almost complex structure $J^T$ compatible with $g_T$. The integrability of $J^{T}$ (as a CR structure of co-dimension $2$) follows from the fact that $V, JV$ are $J$-holomorphic and $F_V$ and $F_{JV}$ are $\mathcal{V}$-basic of type $(1,1)$. Thus, $(g^T, J^T)$ defines a transversal Hermitian structure on the foliated manifold $(M, \mathcal{V})$.
\end{rmk}

\begin{lemma} \label{l:fbasic} Let $(M^{2n},g,J,H)$ be a compact BHE manifold. The soliton potential $f$ guaranteed by Proposition \ref{p:BHEprop} is basic. Furthermore, $f$ is constant if and only if $(\gw, J)$ is Gauduchon, i.e. $d^* \theta =0$.
\end{lemma}
\begin{proof}
    Tracing the soliton equation yields $R - \tfrac{1}{4} \brs{H}^2 + \gD f = 0$.  Since the metric and $H$ are both invariant under $V$, it follows that 
    $$0 = L_V(R - \tfrac{1}{4} \brs{H}^2 + \gD f) = L_V \Delta f = \Delta L_V f.$$
    Since $M$ is compact, it follows from the maximum principle that $L_V f$ is constant, hence zero using again that $M$ is compact.  The same argument applies for $JV$.

    For the second part, as 
    $\eta = \theta - df$ is the $g$-dual of a Killing field $V$, we have
    \[ 0=d^*\eta = d^* \theta - \Delta f. \]
    We conclude again my the maximum principle that $f=const$ iff $d^* \theta =0$.
\end{proof}

\subsection{Reduction of BHE equation} \label{sec: EM soliton}

\emph{ From now on, we assume that $(g,J)$ is  a  {non-K\"ahler} BHE metric  normalized by scale so that the vector fields $V, JV$ are unitary}. We  denote by $(g^T, J^T)$ the transversal Hermitian structure on $\mathcal{H}$ defined in the previous subsection.  We next describe the equations of the BHE system in terms of this transversal structure.   We emphasize that almost the entire discussion is local, and we eventually show a complete local equivalence between the reduced geometric structures and the BHE structure on the total space.  However, since ultimately we are concerned with compact examples, in view of Lemma \ref{l:fbasic}, we will \emph{assume throughout that $f$ is basic}.

\begin{prop}\label{P3}
    Let $(M^{2n},g,J,H)$ be a BHE manifold.   Then one has
    \begin{enumerate}
        \item $\theta_{g^T} = df$,
        \item $\tr_{\omega^T}F_V=0,\ \tr_{\omega^T}F_{JV}= -2$,
        \item $dd^c\omega^T-F_{V}\wedge F_{V}-F_{J V}\wedge F_{J V}=0$.
    \end{enumerate}
\end{prop}
\begin{proof}
    First, we compute the Lee form by (\ref{2.2}). 
    \begin{align*}
        \theta&=\tfrac{1}{2}\tr_{\omega}d\omega = \tfrac{1}{2}\tr_{\omega}\big(F_V\wedge J\eta - \eta\wedge F_{JV}+d\omega^T \big).
    \end{align*}
    As $F$ is supported on $\mathcal{H}$ (see Lemma~\ref{L3}), it follows that
    \begin{align*}
        \theta&=\tfrac{1}{2}\tr_{\omega}\big(F_V\wedge J\eta - \eta\wedge F_{JV}+d\omega^T \big) = \theta_{g^T}+\tfrac{1}{2}(\tr_{\omega^T}F_V) J\eta - \tfrac{1}{2}(\tr_{\omega^T}F_{JV}) \eta.
    \end{align*}
    Subtracting $df$ from both sides we get
    $$ \eta = \theta-df = (\theta_{g^T} - df)+\tfrac{1}{2}(\tr_{\omega^T}F_V) J\eta - \tfrac{1}{2}(\tr_{\omega^T}F_{JV}) \eta. $$
    As $\eta$ and $J\eta$ are orthogonal, and both are orthogonal to the horizontal distribution, and $df$ is horizontal, it follows that
    \begin{align*}
        \theta_{g^T}=df,\quad \tr_{\omega^T}F_V=0,\quad \tr_{\omega^T}F_{JV}=-2,
    \end{align*}
    as claimed in the first two formulas.  Next, we differentiate $\gw = \gw^T + \gw^V$ and use \Cref{L3} to yield
    \begin{align*}
        d\omega&=d\omega^T + F_V\wedge J\eta - \eta\wedge F_{JV},
        \\ d^c\omega&=d^c\omega^T - F_V\wedge \eta - J\eta\wedge F_{JV},
        \\ dd^c\omega&=dd^c\omega^T - F_V\wedge F_V - F_{JV}\wedge F_{JV}.
    \end{align*}
    As $\gw$ is pluriclosed, the final claimed formula follows.
\end{proof}

\begin{rmk} The first equation above says that the transverse Hermitian metric is conformally balanced. 
The second is equivalent to the statement that the connection $\mu$ is Hermitian-Yang-Mills with prescribed central element. 
The final equation is referred to as the Bianchi identity or anomaly cancellation equation due to its appearance in the physical Hull-Strominger system \cite{Hull,StromingerSST}.
\end{rmk}

Next, we show that the transverse geometry satisfies the {\em Einstein-Maxwell soliton equations}.  This is the result of a series of identities for the Bismut curvature.  We fix convenient notation for calculations.
We will use an adapted frame $\{e_1,\ldots,e_{2n-2},V,JV\}$ guaranteed by Lemma \ref{l:frame}.  We will use lowercase Roman letters to refer to the vectors $\{e_i\}$ and Greek letters $e_{\alpha}=V$, ${e_\beta=JV}$.  Furthermore, $F_\alpha=F_V$ and $F_\beta=F_{JV}$.
Finally, we use the upper Roman letters to indicate a general element of the overall frame.  

\begin{lemma}\label{Ch}
    Let $(M^{2n},g,J,H)$ be a BHE manifold.  The nonvanishing Christoffel symbols of the Levi-Civita and Bismut connections are
    \begin{gather} \label{CS2}
    \begin{split}
    \Gamma^g_{ijk}=\Gamma^{g^T}_{ijk},\quad &\Gamma^g_{ij\alpha}= - \tfrac{1}{2}F_{ij\alpha}, \quad 
        \Gamma^g_{\alpha i j}=\Gamma^g_{i\alpha j} = \tfrac{1}{2}F_{ij\alpha}, \quad \Gamma^g_{ij\beta}= - \tfrac{1}{2}F_{ij\beta}, \quad 
        \Gamma^g_{\beta i j}=\Gamma^g_{i\beta j} = \tfrac{1}{2}F_{ij\beta}.\nonumber
    \end{split}
    \end{gather}
    \begin{align}
     \Gamma^B_{ijk}=\Gamma^{g^T}_{ijk} + \tfrac{1}{2}H^T_{ijk}, \quad \Gamma^B_{\alpha ij}=F_{ij\alpha},\quad \Gamma^B_{\beta ij}=F_{ij\beta}.\label{CS}
    \end{align}
\end{lemma}
\begin{proof} 
    A direct application of the Koszul formula 
    yields
    \begin{align*}
        \Gamma^g_{ijk}=\Gamma^{g^T}_{ijk},\quad \Gamma^g_{ij\alpha}= - \tfrac{1}{2}F_{ij\alpha}, \quad 
        \Gamma^g_{i\alpha j}= \Gamma^g_{\alpha i j} = \tfrac{1}{2}F_{ij\alpha}.
    \end{align*}
    To obtain the Christoffel symbols of the Bismut connection, we must add the torsion $3$-form $H$ since $\Gamma^B_{ABC}=\Gamma^g_{ABC}+\tfrac{1}{2}H_{ABC}$.
    From \eqref{eq: torsion form}, the non-vanishing components of $H$ are
    \begin{align*} 
        H_{ijk}=H^T_{ijk},\quad H_{ij\alpha}=F_{ij\alpha},
    \end{align*}
    and so the lemma follows (the computations for the $\gb$ components are equal).
\end{proof}

Next we compute the coefficients of the Bismut curvature tensor.
\begin{lemma}\label{L4}
    Let $(M^{2n},g,J,H)$ be a BHE manifold.  The nonvanishing components of the Bismut curvature tensor are
    \begin{enumerate}
        \item $R^B_{ijkl}=R^{g^T}_{ijkl}+F_{ij\alpha}F_{kl\alpha}+F_{ij\beta}F_{kl\beta}+\tfrac{1}{2}\big( (\nabla_iH^T)_{jkl}-(\nabla_jH^T)_{ikl}+\tfrac{1}{2}H^T_{jkm}H^T_{iml}-\tfrac{1}{2}H^T_{ikm}H^T_{jml} \big).$ 
        \item $R^B_{\alpha\beta ij}=F_{ik\alpha}F_{jk\beta}-F_{ik\beta}F_{jk\alpha}$.
        \item $R^B_{\alpha ijk}=-(\nabla_i^{B}F_\alpha)_{jk}$, $R^B_{\beta ijk}=-(\nabla^{B}_iF_\beta)_{jk}$.
    \end{enumerate}
\end{lemma}
\begin{proof}
    Since $V$ and $JV$ are Bismut-parallel, all cases except (1)-(3) are $0$.
    Now, recall that from the definition, the Bismut Riemannian curvature is given by 
    \begin{align*}
    R^B(e_A,e_B,e_C,e_D)=e_A(\Gamma^B_{BCD})-e_B(\Gamma^B_{ACD})+\Gamma^B_{BCE}\Gamma^B_{AED}-\Gamma^B_{ACE}\Gamma^B_{BED}-[e_A,e_B]^E\Gamma^B_{ECD}.
    \end{align*}
    Using (\ref{CS}), we see that 
    \begin{align*}
        R^B_{ijkl}&=e_i(\Gamma^B_{jkl})-e_j(\Gamma^B_{ikl})+\Gamma^B_{jkE}\Gamma^B_{iEl}-\Gamma^B_{ikE}\Gamma^B_{jEl}-[e_i,e_j]^E\Gamma^B_{Ekl}
        \\&=e_i(\Gamma^{g^T}_{jkl} 
        + \tfrac{1}{2}H^T_{jkl})-e_j(\Gamma^{g^T}_{ikl} 
        + \tfrac{1}{2}H^T_{ikl})
        +\Gamma^B_{jkE}\Gamma^B_{iEl}-\Gamma^B_{ikE}\Gamma^B_{jEl}-[e_i,e_j]^E\Gamma^B_{Ekl}
        \\&= R^{g^T}_{ijkl}+F_{ij\alpha}F_{kl\alpha}+F_{ij\beta}F_{kl\beta}+\tfrac{1}{4}H^T_{jkm}H^T_{iml}-\tfrac{1}{4}H^T_{ikm}H^T_{jml}
        \\&\kern2em+\tfrac{1}{2}\big(e_i(H^T_{jkl})-e_j(H^T_{ikl})+\Gamma^{g^T}_{jkm}H^T_{iml}+\Gamma^{g^T}_{iml}H^T_{jkm}-\Gamma^{g^T}_{ikm}H^T_{jml}-\Gamma^{g^T}_{jml}H^T_{ikm} \big)
        \\&=R^{g^T}_{ijkl}+F_{ij\alpha}F_{kl\alpha}+F_{ij\beta}F_{kl\beta}+\tfrac{1}{2}\big( (\nabla_iH^T)_{jkl}-(\nabla_jH^T)_{ikl}+\tfrac{1}{2}H^T_{jkm}H^T_{iml}-\tfrac{1}{2}H^T_{ikm}H^T_{jml} \big) .
    \end{align*}
    Now, recalling that $F_\alpha$ is a closed form, we compute
    \begin{align*}
        e_\beta\big( F_{\alpha ij} \big)&=(\nabla_\beta F_\alpha)_{ij}+\Gamma^g_{\beta  ik}F_{\alpha kj}+\Gamma^g_{\beta  jk}F_{\alpha ik}
        \\&=-(\nabla_i F_\alpha)_{j\beta }-(\nabla_j F_\alpha)_{\beta  i}+\Gamma^g_{\beta  ik}F_{\alpha kj}+\Gamma^g_{\beta  jk}F_{\alpha ik}
        \\&=(\Gamma^g_{i\beta  k}-\Gamma^g_{\beta  ik})F_{\alpha jk}+(\Gamma^g_{\beta  jk}-\Gamma^g_{j\beta  k})F_{\alpha ik}
        \\&=0,
    \end{align*}
    by (\ref{CS2}). Similarly, we can deduce that $e_\alpha(\Gamma^B_{\beta ij})=e_\alpha(\Gamma^B_{\alpha ij})=e_\beta(\Gamma^B_{\alpha ij})=
    e_\beta(\Gamma^B_{\beta ij})=0$.  Therefore,
    \begin{align*}
        R^B_{\alpha \beta ij }&=e_{\alpha}(\Gamma^B_{\beta ij})-e_{\beta}(\Gamma^B_{\alpha ij})+\Gamma^B_{\beta iE}\Gamma^B_{\alpha Ej}-\Gamma^B_{\alpha iE}\Gamma^B_{\beta Ej}-[e_\alpha,e_\beta]^E \Gamma^B_{Eij}
        \\&=F_{ik\alpha}F_{jk\beta}-F_{ik\beta}F_{jk\alpha}.
    \end{align*}
    Next, 
    \begin{align*}
        R^B_{\alpha ijk}&=e_{\alpha}(\Gamma^B_{ijk})-e_{i}(\Gamma^B_{\alpha jk})+\Gamma^B_{ijE}\Gamma^B_{\alpha Ek}-\Gamma^B_{\alpha jE}\Gamma^B_{i Ek}-[e_\alpha,e_i]^E \Gamma^B_{Ejk}
        \\&=-e_{i}(\Gamma^B_{\alpha jk})+\Gamma^B_{ijl}\Gamma^B_{\alpha lk}-\Gamma^B_{\alpha jl}\Gamma^B_{ilk}
        \\&=-e_{i}(F_{jk\alpha}) +\Gamma^B_{ijl}F_{lk\alpha} + \Gamma^B_{ikl}F_{ jl\alpha}
        \\&=- (\N^B_i F_\alpha)_{jk}.
    \end{align*}
    The equation for $R^B_{\beta i j k}$ follows similarly.
\end{proof}

Finally, using the soliton equations for the total space metric, we show that the transverse metric satisfies the Einstein--Maxwell equations.

\begin{prop} \label{p:EM}
Let $(M^{2n},g,J,H)$ be a BHE manifold.  The transverse geometry satisfies the Einstein--Maxwell equations 
\begin{align*}
    \Rc^{g^T,H^T,f} - F^2=0, \quad dF=0,\quad (d^B)^*_{g^T} F+{\iota_{\nabla f}F}=0, 
\end{align*}
where: 
\begin{enumerate}
\item
$\Rc^{g^T,H^T,f}:=\Rc^{g^T}-\tfrac{1}{4}(H^T)^2+\nabla^2f-\tfrac{1}{2}(d^*_{g^T}H^T)-\tfrac{1}{2}\iota_{\nabla f}H^T $ is the twisted Bakry-Emery curvature, 
\item $F^2(X,Y):= \IP{i_X F_V, i_Y F_V}_{g^T} + \IP{i_X F_{JV}, i_Y F_{JV}}_{g^T}$,
\item $\left((d^B)^*_{g^T} \psi\right)(X) = (d^*_{g^T}\psi)(X) -\tfrac{1}{2}\IP{i_X H^T, \psi}_{g^T}$ is the co-differential of the transversal Bismut connection of $(g^T, J^T)$ acting on basic $2$-forms $\psi$.
\end{enumerate}
Furthermore, one has
\begin{align} \label{Con}
    \tfrac{1}{6}|H^T|^2+\tfrac{1}{2}|F|^2+\Delta f-|\nabla f|^2=\mbox{const}.
\end{align}
\end{prop}
\begin{proof}
Note that BHE metrics satisfy the soliton equations $ {\rm Rc}^B_{AE}=-\nabla^B_A\nabla^B_Ef$. Taking trace of equations in \Cref{L4}, we derive
\begin{align*}
    {\rm Rc}^B_{il}={\rm Rc}^{g^T}_{il}-F^2_{il}-\tfrac{1}{2}(d^*H^T)_{il}-\tfrac{1}{4}(H^T)^2_{il}=-\nabla^B_i\nabla^B_l f,
\end{align*}
and
\begin{align*}
    {\rm Rc}^B_{\alpha k}=(d^B)_{g^T}^*(F_\alpha)_k=-\nabla^B_\alpha\nabla^B_kf={-(\nabla_{\alpha} df)_k} + \tfrac{1}{2}(\iota_{\nabla f}H)_{\alpha k}= {-}(\iota_{\nabla f}F_{\alpha})_{k}.
\end{align*}
The last statement follows by recalling that generalized Ricci solitons satisfy the dilaton equation (\cite[Proposition 4.33] {GRFBook})
\begin{align*}
    \tfrac{1}{6}|H|^2+\Delta f-|\nabla f|^2=\mbox{const},
\end{align*}
and here $|H|^2=|H^T|^2+3|F|^2$.
\end{proof}

Finally we observe that the reduced equations we derived are equivalent locally to a BHE metric with certain symmetries.  Related results have appeared in the literature (cf. \cite[Proposition 3.1] {fino2019astheno} \cite[Proposition 4]{grantcharov2008calabi}).

\begin{cor}\label{c:transversal} Fix $(M^T, g^T, J^T)$ a Hermitian manifold and suppose there exist closed forms $F_V, F_{JV} \in \Lambda^{1,1}$ satisfying that ${\rm span}_{\mathbb{R}}([F_V], [F_{JV}]) \subset H^2(M^T, \mathbb{R})$ admits a basis in $H^2(M^T,  \mathbb Z)$ and a smooth function $f \in C^{\infty}(M^T)$ such that the conclusions of Propositions \ref{P3} and \ref{p:EM} hold for the data $(g^T, H^T, F_V, F_{JV}, f)$.  Then there exists a principal $T^2$-bundle $M \to M^T$ endowed with a principal connection $\mu$ which,  in a suitable basis $\{V, JV\}$ of $\mathfrak{t}^2={\rm Lie}(T^2)$ satisfies $\mu= \eta V + (J\eta) JV$,  $d\eta=F_V, \, \, d(J\eta) =F_{JV}$ for some $1$-forms $\eta$ and $J\eta$ on $M$,  and such that the Hermitian metric on $M$ defined by the K\"ahler form $\gw = \eta \wedge J\eta + \pi^* \gw^T$  and the integrable almost complex structure $J$ given by the horizontal lift of $J^T$ and the tautological action on the vertical fundamental vector fields $V, JV$ is BHE,  with holomorphic Killing vector fields equal to $V, JV$ and transversal K\"ahler geometry  $(g^T, H^T, F_V, F_{JV}, f)$.
\end{cor}
\begin{proof} Let us first consider the case when ${\rm span}_{\mathbb R}([F_V], [F_{JV}])$ is $2$-dimensional. Given the reduced data $(g^T, H^T, F_{V}, F_{JV}, f)$ on $(M^T, J^T)$,  and a basis $\{ \epsilon_1, \epsilon_2\}$ of ${\rm span}_{\mathbb R}([F_V], [F_{JV}])$ in the torsion free part $\widetilde{H}^2(M^T, {\mathbb Z})$ of $H^2(M^T, \mathbb{Z})$, there exist by Chern-Weil theory a principal $T^2$-bundle $\pi: M \to M^{T}$ (unique up to tensoring  with a flat principal $T^2$-bundle over $M^T$) whose torsion-free Chern class is $\epsilon=(\epsilon_1, \epsilon_2)$. 
Let $(\xi_1, \xi_2)$ denote the circle generators of the $S^1\times S^1$-action on the total space $M$, viewed also as a basis of the Lie algebra $\mathfrak{t}^2$ of $T^2$. Writing
\[ 2\pi \epsilon_1 = a_{V} [F_V] + a_{JV} [F_{JV}], \qquad 2\pi \epsilon_2 = b_V[F_V] + b_{JV}[F_{JV}],   \]
for  $(a_V, a_{JV}), (b_V, b_{JV})\in {\mathbb R}^2$, 
$M$ has a unique connection $\mu$ whose curvature is 
 \[ d\mu=\Big(a_{V} (\pi^* F_V) + a_{JV} (\pi^*F_{JV})\Big)\xi_1 +  \Big(b_V( \pi^*F_V) + b_{JV} (\pi^* F_{JV})\Big) \xi_2.\] The connection $\mu$ defines a splitting $TM= \mathcal{V} \oplus \mathcal{H}$ into vertical and horizontal parts.  We now define vector fields $(V, JV)$ in $\mathcal{V}$ by inverting  the relation
\[V= a_V \xi_1  + b_V \xi_2, \qquad JV= a_{JV}\xi_1 + b_{JV} \xi_2.   \] 
We then define a  Riemannian metric $g$ and a $g$-orthogonal almost complex structure $J$  on $M$ by pulling back $(g^T, J^T)$ to $\mathcal{H}$, and defining orthogonally $(g, J)$ on $\mathcal{V}$ so that ${V, JV}$ is a unitary Hermitian basis.  It follows from the proofs of Propositions \ref{P3} and \ref{p:EM} that the Hermitian metric  $(g, J)$ is BHE, as required.

\smallskip 
In the case when when ${\rm span}_{\mathbb R}([F_V], [F_{JV}])$ is $1$-dimensional and $[F_V]\neq 0 \neq [F_{JV}]$, we take two integer generators  $\epsilon_1, \epsilon_2 \in \widetilde{H}^{2}(M^T, \mathbb{Z})$ of  ${\rm span}_{\mathbb R}([F_V], [F_{JV}])$ and consider the principal $T^2$-bundle over $M^T$ with Chern calss $\epsilon=(\epsilon_1, \epsilon_2)$ and connection $\mu$ whose curvature is
\[ d\mu= a_V (\pi^*F_V) \xi_1 + b_{JV} (\pi^* F_{JV})\xi_2\] where the non-zero real constant $a_V$ and $b_{JV}$ are defined from
\[ a_V [F_V] = 2\pi \epsilon_1, \qquad  b_{JV}[F_{JV}] = 2\pi \epsilon_2.\]
We conclude as in the previous case, letting $V={a_V} \xi_1$ and $JV ={b_{JV}} \xi_2$.

\smallskip
Finally, if one class (say $[F_V]=0$) or both classes $[F_V]$ and $ [F_{JV}]$ are zero, we consider,  respectively,  principal $T^2$-bundles over $M^T$ with Chern class
$\epsilon=(0, \epsilon_2)$ or $\epsilon=(0,0)$, and let (respectively)  $V=\xi_1, JV =(2\pi\epsilon_2/[F_{JV}]) \xi_2$  and  $V=\xi_1, JV=\xi_2$ in the construction above. \end{proof}

\begin{rmk} If $(g^T, H^T, F_{V}, F_{JV}, f)$ satisfies the conditions of  Propositions \ref{P3} and \ref{p:EM} so does $(g^T, H^T, -F_V, F_{JV})$. Geometrically, this corresponds to the fact that almost complex structure $I$  on $M$ which is equal to $-J$ on $\mathcal{V}$ and to $J$ on $\mathcal{H}$ is Hermitian with respect to $g$ is also BHE. 
\end{rmk}

\begin{rmk} For a given reduced data $(g^T, H^T, F_{V}, F_{JV}, f)$ on $(M^T, J^T)$, Corollary~\ref{c:transversal} defines infinitely many diffeotypes on BHE manifolds $M$, all sharing the same universal cover, and parametrized by the choices of bases of ${\rm span}_{\mathbb R}(F_V, F_{JV})$ inside the torsion free part of $H^2(M^T, \mathbb{Z})$ and pairs of torsion elements of $H^2(M^T, \mathbb{Z})$.
\end{rmk}

\subsection{The transversal K\"ahler geometry for BHE threefolds} 

When $(M^6, g, J)$ is a non-K\"ahler BHE complex threefold, the transversal Hermitian structure $(g^T, J^T)$ is real four-dimensional,  and Propositions~\ref{P3} and \ref{p:EM} admit further simplification which we now present.

{
\begin{prop}\label{p:4D}  Let $(M^T, g^T, J^T)$ be any local leaf space for $(M, \mathcal{V})$ with the induced transversal Hermitian structure. Then, $g^T$ is conformal to a K\"ahler metric $(g_K^T, J^T)$ which satisfies the following conditions:
\begin{enumerate}
\item $\Scal(g_K^T) = 2e^{f} > 0$.
\item There exists a closed primitive real $(1,1)$-form $\alpha^T$ such that
\[ \left(dd^c \Scal(g_K^T)\right) \wedge \omega_K^T = 2\left( \alpha^T \wedge \alpha^T + \rho_K^T\wedge \rho_K^T \right), \]
where $\omega_K^T$ is the K\"ahler form and $\rho_K^T$ is the Ricci form of $(g_K^T, J^T)$.
\end{enumerate}
Conversely, any K\"ahler surface $(M^T, J^T, g_K^T)$ satisfying the above conditions for a closed primitive $(1,1)$-form $\alpha^T$ such that ${\rm span}_{\mathbb R}([\alpha^T], c_1(M^T, J^T))$ admits a basis in $H^2(M^T, \mathbb{Z})$ gives rise, via Corollary~\ref{c:transversal}, to a BHE threefold by setting 
\[ g^T:= \frac{\Scal(g_K^T)}{2}g_K^T, \qquad F_V := \alpha^T, \qquad F_{JV}:=- \rho_K^T, \qquad f= \log\left(\frac{\Scal(g_K^T)}{2}\right).\]
\end{prop}
\begin{proof} We first recall that in real 4-dimensions 
\begin{equation} \label{eq:Lee} d\omega^T = \theta_{g^T} \wedge \omega^T. \end{equation}Proposition~\ref{P3}-(1) then means that
\[ g_K^T := e^{-f} g^T\]
is a K\"ahler metric.  Equivalently, the Bismut torsion of $(g^T, J^T)$ has the form
\begin{equation}\label{eq:torsion} H^T = -d^c\omega^T = { -}Jdf \wedge \omega^T.\end{equation}
Notice using Proposition~\ref{P3}-(2) that we can define primitive $(1,1)$ forms $\ga^T$ and $\gamma^T$ via 
\[ F_V = \alpha^T, \quad F_{JV} = - \tfrac{1}{2} \omega^T + \gamma^T.\]
Thus the symmetric part of the Einstein--Maxwell equation in Proposition~\ref{p:EM} reads
\begin{equation}\label{EM-4D}
{\rm Rc}^{g^T} - \tfrac{1}{4}(H^T)^2+ \nabla^2f = F^2 = \tfrac{1}{4}\left(|\alpha^T|^2 + |\gamma^T|^2 +1\right)g^T + \gamma^T J^T.
\end{equation}
Notice that, by \eqref{eq:torsion}, 
\[ (H^T)^2 (X, Y) = \IP{i_X H^T, i_Y H^T}= 2|df|^2g(X,Y) - 2 df \otimes df, \]
whereas Proposition~\ref{P3}-(3),  \eqref{eq:Lee}  and Proposition~\ref{P3}-(1) give 
\begin{equation}\label{eq:laplace} dd^c \omega^T = (\Delta f) dV_{g^T} = \tfrac{1}{2}\left( -|\alpha^T|^2 -|\gamma^T|^2 +1\right) dV_{g^T}. \end{equation}
Using the above two identities to express the terms $(H^T)^2$ and $|\alpha^T|^2 + |\gamma^T|^2$ in \eqref{EM-4D}, we can  rewrite  \eqref{EM-4D} as
\[ {\rm Rc}^{g^T} + \nabla^2 f + \tfrac{1}{2} df\otimes df +\tfrac{1}{2}(\Delta f - |df|^2)g = \tfrac{1}{2}g^T + \gamma^T J. \]
The LHS is nothing but the Ricci tensor of $g_K^T= e^{-f}g^T$ (see e.g. \cite[Ch.1-J]{Besse}). Thus, by composing with $J^T$,  we equivalently get
\begin{equation}\label{rho} \rho_K^T = \tfrac{1}{2}\omega^T - \gamma^T = -F_{JV}.\end{equation}
Taking trace with respect to $\omega_K^T= e^{-f} \omega^T$ in the above equality gives 
\begin{equation}\label{eq:scal}
\Scal(g_K^T)= 2 e^{f},\end{equation}
finishing the proof of item (1).
Substituting back \eqref{rho} in Proposition~\ref{P3}-(3) and using that now $\omega^T = \frac{\Scal(g_K^T)}{2} \omega_K^T$, we obtain the claimed item (2).

\vskip 0.1in
We now show that the conditions expressed in terms of $(g_K^T, J^T, \alpha^T)$ are actually sufficient.  In particular we verify that the data $(g^T, J^T, F_V = \ga^T, F_{JV} = - \rho_K^T,f)$ satisfy all the conditions of Propositions~\ref{P3} and \ref{p:EM}, and then invoke Corollary~\ref{c:transversal}.  First note that $\ga^T$ is primitive with respect to $\omega^T$, and hence also $\omega_K^T$.  As $\ga^T$ is also closed, it is thus a harmonic anti-self-dual $2$-form with respect to each metric in the conformal class $[g^T]$.  It is not difficult to see using \eqref{eq:torsion} and $\tr_{\omega^T}(\ga^T)=0$ that $(d^B)^*_{g^T} \ga^T {+} \imath_{\nabla f} \ga^T = d^*_{g^T} \ga =0$.  Next, defining also $\gamma^T$ via \eqref{rho}, note that $\gamma^T$ is an anti-self-dual form.  As furthermore $\omega^T$ is self-dual, the fact that $\rho_K^T$ is closed is equivalent to
\[ d^*_{g^T} \gamma^T = -\tfrac{1}{2} d^*_{g^T} \omega^T.\]
As $\gw^T$ is conformally K\"ahler, we also note 
\[ d^*_{g^T} \omega^T = -J\theta_{g^T} = {-}Jdf, \]
and also that \eqref{eq:torsion} holds.  One then directly checks that
{
\[ 
\begin{split} \left((d^B)^*_{g^T} F_{JV}\right)_X &=  \left(d^*_{g^T} F_V\right)_X -\frac{1}{2}\left\langle \imath_X H^T, F_{JV}\right\rangle \\
&=  Jdf(X) +\frac{1}{2}\left\langle \imath_X(Jdf \wedge \omega^T), F_{JV} \right\rangle
\\ &=- \imath_{\nabla f} F_{JV},
\end{split}\]
where we have used $\tr_{\omega^T}(F_{JV})=-2$ for passing to the last line.
}
Furthermore we observe that the skew part of ${\rm Rc}^{g^T, H^T, f}$,
\[ -\tfrac{1}{2} (d^*_{g^T}H^T) -\tfrac{1}{2}\imath_{\nabla f} H^T,\]
 will identically vanish by noting that \eqref{eq:torsion} implies the vanishing of the two terms separately.  Next we check \eqref{Con}. By \eqref{eq:torsion}, we have $|H^T|^2 = 6|\nabla f|^2$ whereas by the assumed equation of item (2), we compute
\[ |F|^2=|F_V|^2 + |F_{JV}|^2 = 1 +|\gamma^T|^2 + |\alpha^T|^2 = -2\Delta f + 2. \]
With these equations in place one can reverse the arguments from the first part to establish the symmetric part of the Einstein-Maxwell equation as in line \eqref{EM-4D}.
\end{proof}
}

\begin{rmk}\label{r:PDE} Proposition~\ref{p:4D} reduces the search for compact, non-K\"ahler  BHE $6$-dimensional manifolds $M$ for which the corresponding foliation $\mathcal{V}$ is regular (i.e. the space of leaves is a smooth compact manifold) to solving a non-linear PDE problem on a given compact K\"ahler surface $S=(M^T, J^T)$  satisfying the  following apriori conditions on its first Chern number and Kodaira dimension 
\[ c_1^2(S) \geq 0, \qquad {\rm Kod}(S)=-\infty,\]
where the first inequality follows by integrating the relation in Proposition~\ref{p:4D}-(2) whereas the condition on the Kodaira dimension is a consequence from Proposition~\ref{p:4D}-(1). 

On any such $S$, given a K\"ahler class $\Omega \in H^{1,1}(S, \mathbb{R})$ and a cohomology class $A \in H^2(S, \mathbb{R})$ satisfying
\begin{equation}\label{eq:PDE-topo}
 \Omega \cdot A=0,  \qquad A\cdot A =-c_1^2(S) , \qquad {\rm span}_{\mathbb{R}}(A, c_1(S)) =\Lambda \otimes \mathbb{R} \qquad \Lambda < H^2(S, \mathbb{Z}),\end{equation}
we want to find a K\"ahler metric $\omega_{\varphi}= \omega_0 + dd^c \varphi \in \Omega$ such that 
\begin{equation}\label{eq:PDE} \tfrac{1}{2} (dd^c  \Scal(\omega_{\varphi})) \wedge \omega_{\varphi} =  \alpha_{\varphi} \wedge \alpha_{\varphi} + \rho(\omega_{\varphi}) \wedge \rho(\omega_{\varphi}), \end{equation}
where $\rho(\omega_{\varphi})=-\tfrac{1}{2} dd^c \log\left(\frac{\omega_{\varphi}^2}{Vol_{\mathbb C}}\right)$ is the Ricci form of $\omega_\varphi$, $\Scal(\omega_{\varphi}) = 4\left(\frac{\rho(\omega_{\varphi})\wedge \omega_{\varphi}}{\omega_{\varphi}\wedge \omega_{\varphi}}\right)$ is its scalar curvature,   and $\alpha_{\varphi}$ is the harmonic representative of $2\pi A$ with respect to $\omega_{\varphi}$. Notice that as the harmonic $2$-form $\alpha_{\varphi}$ has a constant trace with respect to  the K\"ahler metric $\omega_{\varphi}$, the topological condition  \eqref{eq:PDE-topo} and the fact that $b_+(S)=1$ when ${\rm Kod}(S)=-\infty$ guarantee that $\alpha_{\varphi}$ is a primitive $(1,1)$-form.  It will be interesting to develop a comprehensive existence theory for  the PDE \eqref{eq:PDE} for the unknown function $\varphi$.  Moreover, one expects these formal considerations to apply to the case of a general foliation.
\end{rmk}

\subsection{The transversal Betti numbers}

By Remark~\ref{r:riem-fol}, $(M, \mathcal{V}, g^T)$ is a compact manifold endowed with a Riemannian foliation (see \cite{Molino}) and transversal orientation determined by $J^T$. One can then define Hodge theory for ${\mathcal V}$-basic forms on $N$, by using the Hodge Laplace operators with respect to transversal metric $g^T$ (see \cite{KH}).   This leads to finite dimensional basic deRham cohomology groups $H^{k}_T(M, \mathcal{V}) \cong {\mathcal H}^{k}_T(M, \mathcal{V})$,   where $\mathcal{H}^{k}_T(M, \mathcal{V})$ denotes the space of basic harmonic $k$-forms with respect to $g^T$. We denote by $b_k^T(M, \mathcal{V}):= {\rm dim}_{\mathbb{R}} {\mathcal H}^{k}(M, \mathcal{V})$ the corresponding transversal Betti numbers.  

When $M$ is real 6-dimensional, the transversal Hodge  theory is modeled on a 4-dimensional space, so we have additionally the splitting \[\mathcal{H}_T^2(M, \mathcal{V}) = \mathcal{H}_T^+(M, \mathcal{V}) \oplus \mathcal{H}_T^-(M, \mathcal{V})\]
as a sum of self-dual and anti-self-dual harmonic basic $2$-forms, thus leading to the definition of the Betti numbers $b^T_{\pm}(M, \mathcal{V})= {\rm dim}_{\mathbb{R}}{\mathcal H}^{\pm}(M, \mathcal{V})$ with 
\[ b_2^T(M, \mathcal{V})= b^T_+(M, \mathcal{V}) + b^T_-(M, \mathcal{V}).\]

\begin{prop}\label{p:Betti} Let $(M, \omega, J)$ be a compact non-K\"ahler real 6-dimensional BHE manifold. Then $b^T_{+}(M, \mathcal{V})=1$, $b^T_-(M,\mathcal{V}) \geq 1$,  and thus $b_2(M, \mathcal{V}) \geq 2$.
\end{prop}
\begin{proof} By Proposition~\ref{p:4D}, $\omega_K$ is a non-zero harmonic basic self-dual $2$-form on $M$ so that $b^T_+(M, \mathcal{V})\geq 1$. To see that $b^T_+(M, \mathcal{V})=1$, we use that $\omega_K$ has positive scalar curvature and the transverse version of the well-known Bochner argument: For any basic $g_K^T$-harmonic self-dual $2$-form $\psi$,  we have 
\begin{equation}\label{eq:weitzenbock} (\nabla^{g_K^T})^* \nabla^{g_K^T} \psi + P_{g_K^T}(\psi)=0,\end{equation}
where $P_{g_K^T} = \frac{\Scal(g_K^T)}{3} I - 2 W^+(g_K^T)$; for a K\"ahler surface, the self-dual Weyl tensor  $W^+(g_K^T)$  has eigenvalues $\left(\frac{\Scal(g_K^T)}{6}, - \frac{\Scal(g_K^T)}{12}, -\frac{\Scal(g_K^T)}{12}\right)$ and the K\"ahler form $\omega_K^T$ is an eigenform corresponding to the eigenvalue $\frac{\Scal(g_K^T)}{6}$ (see e.g. \cite{gauduchon-besse4}). It thus follows that when $\Scal(g_K^T)>0$, $P_{g_K^T}$ is a non-negative symmetric operator on $\Lambda^+(\mathcal{H}^*)$, such that $P_{g_T^K}(\psi)=0$  at a point  iff  $\psi$ is a multiple of $\omega_K^T$  at that point. Taking inner product with $\psi$ with respect to $g_K^T$ in \eqref{eq:weitzenbock}, we get 
\[\tfrac{1}{2} \Delta_{g_k^T}|\psi|^2 = |\nabla^{g_K^T} \psi|^2 + \IP{P_{g_K^T}(\psi), \psi}  \geq 0. \]
By the maximum principle, $|\psi|^2$ is constant, $\nabla^{g_J^T} \psi \equiv 0$, 
and $P_{g_K^T}(\psi)\equiv 0$. It thus follows that $\psi = \lambda \omega_K^T$ for some constant $\lambda$. 

We now prove that $b^T_-(M, \mathcal{V}) \geq 1$.
Suppose for contradiction that $b^T_-(M, \mathcal{V})=0$.  Then, in Proposition~\ref{p:4D} we must have $\alpha^T=0$ and
\[ \rho_K^T = \lambda \omega_K^T + d \beta^T,\]
for some basic $1$-form $\beta^T$ on $(M, \mathcal{V})$. The Einstein-Maxwell equation in Proposition~\ref{p:4D} then becomes
\[ dd^c \Scal(g^T_K) \wedge \omega_K^T = 2(\lambda \omega_K^T + d\beta^T)^2.\]
We next multiply the above equality of basic $4$-forms by the $2$-form $\eta\wedge J\eta$ (see \eqref{eq:eta}) and integrate over $M$, recalling that, by Proposition~\ref{p:4D}, $d(-J\eta)=\rho_K^T$, $d\eta = \alpha^T = 0$,  and
$\omega_K^2 \wedge \eta \wedge J\eta$ is a positive volume form on $M$. Using  Stokes Theorem  we get $\lambda =0$, i.e. $\rho_K^T = d\beta^T$.  It then follows that
\begin{equation*}\label{eq:middle} \Scal(g_K^T) \omega_K^T \wedge \omega_K^T = 4 \omega_K^T \wedge \rho_K^T= 4\omega_K^T \wedge d\beta^T.\end{equation*}
Multiplying  the above equality by $\eta\wedge J\eta$ and using that 
$\Scal(g_K^T) dV_{g_K^T} \wedge \eta \wedge J\eta$ is a positive volume form on $M$, we obtain a contradiction:
\[0<\int_M \Scal(g_K^T) \omega_K^T \wedge \omega_K^T \wedge  \eta \wedge J\eta = 4\int_M \omega_K^T \wedge d\beta^T \wedge \eta \wedge J\eta =0, \] 
where for the last equality we integrated by parts and accounted for the degree of basic forms. 
\end{proof}

\begin{prop}\label{p:full-Betti} Let $(M, \gw, J)$ be a non-K\"ahler compact BHE manifold. Then \[h^{1,1}_{\rm BC}(M) \geq b_2^T(M, \mathcal{V}).\] If $\gw$ has constant soliton potential, or if $b_1(M)=0$, we have \[ h^{1,1}_{\rm BC}(M)=b^T_2(M, {\mathcal V}).\]
\end{prop}

\begin{proof}
The Killing vector fields $V, JV$ generate a compact torus $G \subset {\rm Isom}(M, \omega, J)$.  It is well-known that the $G$-averaging map for forms
gives rise to an isomorphism between  the usual deRham complex and the deRham complex for $G$-invariant forms on $M$. As $G$ is the closure of a $\mathbb C$-action which preserves the Hermitian metric and the complex structure, by \cite{Kle}, Proposition~3.1,  the  $G$-averaging map also induces an isomorophism  between the Bott--Chern cohomology  and the $G$-invariant Bott--Chern cohomology. We will thus assume that elements of the $k$-th deRham cohomology groups $H^k(M, {\mathbb R})$ and the Bott--Chern groups $H^{p,q}_{\rm BC}(M, \mathbb{C})$ are represented by $G$-invariant forms.

With this in mind, we are going to define an exact sequence
\begin{equation}\label{eq:ES} 0 \longrightarrow H^{1,1}_{\rm BC}(M, \mathcal{V}, \mathbb{R}) \xrightarrow{\, \, \, \, j \,  \, \, \, } H^{1,1}_{BC}(M, \mathbb{R}) \xrightarrow{\imath_{V^{0,1}}} H_{\rm BC}^{1,0}(M, \mathcal{V})  \end{equation}
as follows. The map $j$ is induced by the embedding from the space of closed basic forms to the space of closed forms on $M$ and the map $\imath_{V^{0,1}}$ is induced from the interior product of a $(1, 1)$ form with the $(0,1)$ vector field $V^{0,1} = \tfrac{1}{2}(V + \sqrt{-1} JV)$: the fact that we consider $G$-invariant forms guarantees $\imath_{V^{0,1}}$ maps $H^{1,1}_{\rm BC}(M, \mathbb{R})$ to $H_{\rm BC}^{1,0}(M, \mathbb{C})$. In order to see that the image of $H^{1,1}_{\rm BC}(M, \mathbb{C})$ under $\imath_{V^{0,1}}$ is in $H^{1,0}_{\rm BC}(M, \mathcal{V})$, we have to show that for any closed real $G$-invariant $(1,1)$-form $\psi$ on $M$,  $\psi(V^{0,1}, V^{1,0})= -\frac{\sqrt{-1}}{2}\psi(V, JV)$ is zero.

As $\imath_{V^{0,1}}\psi$ is  closed (and hence holomorphic) $(1,0)$-form ($\psi$ is closed and $G$-invariant), and as $V^{1,0}$ is a holomorphic vector field, $\psi(V^{0,1}, V^{1,0})$ is a constant. To show that this constant is zero, we use that, by Proposition~\ref{p:4D}, the $1$-forms $\eta$ and $J\eta$ defined in \eqref{eq:eta} satisfy $d(-J\eta)=\rho_K^T$, $d\eta = \alpha^T$  and 
that $\omega_K^T \wedge d(-J\eta) \wedge \eta \wedge J\eta = \frac{\Scal(g_K^T)}{2} dV_{g_K^T} \wedge \eta \wedge J\eta$ is a positive volume form on $M$ whereas $\omega_K^T \wedge d\eta = \omega_K^T \wedge \alpha^T=0$ (as $\alpha^T$ is primitive). We then have 
\[ \begin{split} 
0 &=\imath_{JV} \imath_{V}\left(\psi \wedge \omega_{K}^T \wedge d(-J\eta) \wedge J\eta\right) \\
&= \psi(V, JV) \left(\omega_{K}^T \wedge d(-J\eta) \wedge J\eta\right) + (\imath_{V}\psi) \wedge \omega_K^T\wedge d(-J\eta). 
\end{split}\]
Multiplying the above equality by $\eta$ and integrating over $M$ we get
\[ \begin{split}
&\psi(V, JV)\int_M \omega_{K}^T \wedge d(-J\eta) \wedge \eta \wedge J\eta =\int_M(\imath_{V}\psi) \wedge \omega_K^T\wedge d(-J\eta)\wedge \eta \\
&= \int_M (\imath_V\psi) \wedge \omega_K^T \wedge \alpha^T \wedge J\eta =0,
\end{split}\]
where we have used that $\imath_V \psi$ is closed and integration by part to obtain the second equality, and the fact that $\ga^T$ is primitive to get the final equality.  It thus follows that $\psi(V, JV)=0$, i.e. $\imath_{V^{0,1}}\psi$ is a basic closed $(1,0)$-form.

As any real $G$-invariant $(1,1)$-form on $M$ is in the kernel of $\imath_{V^{0,1}}$ if and only it is $\mathcal{V}$-basic, we obtain the exactness of the maps $j$ and $\imath_{V^{0,1}}$.  Finally, we show that $j$ is injective.  Given a basic $(1,1)$-form $\psi^T$ (which is $G$-invariant by definition) which is $dd^c$-exact on $M$, i.e. $\psi^T = dd^c \varphi$, we can overage $\varphi$ over $G$ to yield that $\psi^T$ is exact in $H^{1,1}_{\rm BC}(M, \mathcal{V})$.  

\smallskip
Now that the exactness of \eqref{eq:ES} is established, we show that 
$H^{1,1}_{\rm BC}(M, \mathcal{V}, \mathbb{R}) \cong \mathcal{H}^2_T(M, \mathcal{V})$. To this end, let  $\psi^T$ be a closed basic $(1,1)$-form.  By the transversal Hodge theorem, $\psi^T = \psi^T_H + d\beta^T$ where $\psi^T_H$ is harmonic with respect to $g_K^T$,  and $\beta^T$ is a basic $1$-form. Any local leaf space $M^T$ is K\"ahler manifold, so the K\"ahler identities for the transversal Hodge Laplacian hold. We thus get from  Proposition~\ref{p:Betti} 
\[ \mathcal{H}_T^+(M, \mathcal{V})\otimes \mathbb{C} = \mathcal{H}^{2,0}_T(M, \mathcal{V}) \oplus \mathcal{H}^{0,2}_T(M, \mathcal{V}) \oplus \mathbb{C}(\omega_K^T) = \mathbb{C}(\omega_K^T),\]
so that 
\[ \mathcal{H}_T^2(M, \mathcal{V})= \mathcal{H}_T^{1,1}(M, \mathcal{V}, \mathbb{R}).\]
It follows that $\psi^T_H$ is a harmonic  basic $(1,1)$-form, and thus the $(0, 1)$-part $(\beta^T)^{0,1}$ of $\beta^T$ is a $\bar\partial$-closed basic form. Using again the K\"ahler identities for the Laplace operator of $g_K^T$,   we have a transversal Hodge decomposition 
\[ (\beta^T)^{0,1} = (\beta^T)_H^{0,1}  + \bar \partial \varphi,  \]
which leads to the identification
\begin{equation}\label{4D-BC} \mathcal{H}^2_{T}(M, \mathcal{V}) = \mathcal{H}_T^{1,1}(M, \mathcal{V}, \mathbb{R}) \cong H^{1,1}_{\rm BC}(M, \mathcal{V}, \mathbb{R}).\end{equation}
From \eqref{eq:ES} and \eqref{4D-BC} we conclude
\[ h^{1,1}_{\rm BC}(M) \geq b_2^T(M, \mathcal{V}).\]

\smallskip
We now show that we have equality in the case when the soliton potential is constant.  To this end, we need to show that the image of $\imath_{V^{0,1}}$ in \eqref{eq:ES} is $0$. This is obvious  if  $H^{1,0}_{\rm BC}(M, \mathcal{V})=0$.  We thus consider the case $H^{1,0}_{\rm BC}(M, \mathcal{V}) \neq 0$.  We shall use the fact that when $f=const$, the transversal Hermitian metric $g^T$ is K\"ahler, so we can take in Proposition~\ref{p:4D} $g_K^T = g^T$ and $\Scal(g^T)= 2$. The Einstein-Maxwell condition of Proposition~\ref{p:EM} now reads
\[ {\rm Rc}^{g^T} = F^2 \geq 0,\]
and ${\rm Rc}^{g^T} \neq 0$ as $\Scal(g^T)=2$. It thus follows by Bochner's formula that for any harmonic basic $1$-form $\beta_H^T$, 
\[ \tfrac{1}{2} \Delta_{g^T}|\beta_H^T|^2= |\nabla^{g^T} \beta_H^T|^2 + {\rm Rc}^{g^T}((\beta_H^T)^{\sharp}, (\beta_H^T)^{\sharp}) \geq 0.\]
By the maximum principle, we conclude that $|\beta_H^T|^2_{g^T}=const$, $\nabla^{g^T}\beta_H^T=0$ and ${\rm Rc}^{g^T}((\beta_H^T)^{\sharp}, \cdot)=0$. The semi-positivity of the Ricci tensor of the K\"ahler metric $g^T=g_K^T$ and $\Scal(g^T)=2$ thus yield  
\[ h^{1,0}_{\rm BC}(M, \mathcal{V}) = 1, \] 
as well as the equality
\begin{equation}\label{eq:+} \beta_H^T \wedge J\beta_H^T \wedge \rho_K^T = \tfrac{1}{2}|\beta_H^T|_{g^T}^2dV_{g_T}.\end{equation}
Furthermore, as $\rho_K^T$ is degenerate, we have $\rho_K^T\wedge \rho_K^T=0$ and therefore Proposition~\ref{p:4D}(2) gives
\begin{equation}\label{eq:0} d\eta = \alpha^T =0.\end{equation}
We now take $\beta = \mathfrak{Re}(\imath_{V^{0,1}}\psi)$ with $\psi$ being a closed $G$-invariant $(1,1)$-form on $M$ and assume $\beta \neq 0$. As $(\imath_{V^{0,1}}\psi)$ is a basic closed $(1,0)$-form on a (locally defined) K\"ahler surface, it is also harmonic with respect to any K\"ahler metric. Thus, we will have by \eqref{eq:+}
\begin{equation}\label{>0} \int_M \beta \wedge J\beta \wedge d(-J\eta) \wedge \eta \wedge J\eta = \tfrac{1}{2}|\beta|^2_{g^T}\int_M dV_{g} >0. \end{equation}
On the other hand, using that $d(-J\eta)=\rho_K^T$ is basic and $\psi(V, JV)=0$ (which we have shown already), we compute
\[ \begin{split}
    0&= \imath_V \imath_{JV}\left(\psi\wedge \psi \wedge d(-J\eta) \wedge J\eta\right) \\
    &= 2 (\imath_V \psi)\wedge (\imath_{JV}\psi)\wedge d(-J\eta) \wedge J\eta  + 2 (\imath_{V}\psi) \wedge \psi \wedge d(-J\eta).
    \end{split}\]
Multiplying with $\eta$  and integrating over $M$ gives 
\[ 
\begin{split} &\int_M \beta \wedge J\beta \wedge d(-J\eta) \wedge \eta \wedge J\eta =\int_M (\imath_V \psi)\wedge (\imath_{JV}\psi)\wedge d(-J\eta) \wedge \eta\wedge J\eta \\
&= \int_M (\imath_{V}\psi) \wedge \psi \wedge d(-J\eta) \wedge \eta =0,\end{split} \]
where for the last equality we have used integration by parts and \eqref{eq:0}. This is a contradiction with \eqref{>0}.
\smallskip
We finally show that $h^{1,1}_{\rm BC}(M)=b_2^T(M, \mathcal{V})$ when 
$b_1(M)=0$. Indeed, in this case we have $b_1^T(M, \mathcal{V})=0$ as the map 
\[ H^{1}_{T}(M, \mathcal{V}) \xrightarrow{\, \, \, \, j \,  \, \, \, } H^{1}(M, \mathbb{R})\]
is injective by a straightforward averaging argument.  This implies that $H^{1,0}_{\rm BC}(M, \mathcal{V})=0$, and hence by \eqref{eq:ES}, $H^{1,1}_{\rm BC}(M, \mathcal{V}, \mathbb{R}) = H^{1,1}_{\rm BC}(M, \mathbb{R})$; we conclude again by \eqref{4D-BC}.
\end{proof}

To end this section, we record the proof of Theorem~\ref{t:mainthm}, which is already contained in the discussion above.

\begin{proof}[Proof of Theorem~\ref{t:mainthm}] Assuming $(M, \gw, J)$ is a compact non-K\"ahler BHE 6-manifold, we rescale $\gw$ so that $|V|=1$ in Proposition~\ref{p:BHEprop}. We can thus apply the conclusions of \S~\ref{sec: EM soliton}. Theorem~\ref{t:mainthm}-(1) is thus established in Proposition~\ref{p:4D}. Theorem~\ref{t:mainthm}-(2) follows from Lemma~\ref{l:fbasic} and Proposition~\ref{p:4D} whereas Theorem~\ref{t:mainthm}-(3) follows from Propositions~\ref{p:Betti} and \ref{p:full-Betti}. \end{proof}

\section{Proof of Corollary~\ref{c:rigiditycor}}

 In this section we complete the proof of Corollary \ref{c:rigiditycor}. Recall that, assuming $(M, \omega, J)$ is a non-K\"ahler BHE $6$-manifold,  the soliton potential is constant iff $(\gw, J)$ is Gauduchon (Lemma~\ref{l:fbasic}), and also in the setup of Proposition~\ref{p:4D}, this happens iff the transversal K\"ahler metric has a positive constant scalar curvature $\Scal(\omega_K^T)$.  Thus, in this case $\rho_K^T = \frac{\Scal(\omega_K^T)}{4}\omega_K - \gamma^T$ for a harmonic anti-self-dual  $2$-form $\gamma^T$. By Propositions~\ref{p:Betti} and \ref{p:full-Betti}, the assumption $h_{\rm BC}^{1,1}(M)=2$  in Corollary~\ref{c:rigiditycor}-(1) is equivalent to 
$b_-^{T}(M, \mathcal{V})=1$,  so that the harmonic anti-self-dual $2$-forms $\alpha^T$ and $\gamma^T$ are linearly dependent.  Similarly, the geometric assumption of Corollary~\ref{c:rigiditycor}-(2) is equivalent to $\Scal(\omega_K)=const$ (as $d^*\theta=0$ see Lemma~\ref{l:fbasic})  and  $\alpha^T=0$ (as $F_V = d\eta = d\theta=0$) whereas the one of Corollary~\ref{c:rigiditycor}-(3) is equivalent to $\Scal(\omega_K)=const$ and $\gamma^T=0$. In each case, $\alpha^T$ and $\gamma^T$ are proportional, and we are going to prove in Lemma~\ref{LP4} below that in this case $\brs{F_V}$ and $\brs{F_{JV}}$ are constant.  The strategy is then to use ideas of \cite{VTA} to show that $F$ is in fact transverse Levi-Civita parallel.  This then will imply that the total space torsion is Bismut-parallel, thus finishing the proof by 
\cite{brienza2024cyt} (cf. also \cite{barbaro2024pluriclosed, ZhengBismutflat})
\subsection{Curvature identities for BHE threefolds}

In this subsection we derive a series of curvature identities for BHE threefolds with trivial soliton potential.  Before beginning our computations we record a lemma deriving a useful frame for computations:

\begin{lemma} \label{l:frame} Let $(M^{2n},g,J)$ be a BHE manifold. For each point $p$, there exists a set of vector fields $\{e_1,...,e_{2n-2}\}$ such that
\begin{enumerate}
    \item $\spn\{e_i\}=\mathcal{H}$,
    \item $Je_{2i-1}=e_{2i},\ 1 \leq i \leq n-1$,
    \item $[e_i,V]=[e_i,JV]=0,$
    \item $[e_i,e_j]=-F_{V}(e_i,e_j)V-F_{J V} (e_i,e_j)JV$.
    \end{enumerate}
\end{lemma}
\begin{proof} The construction of the frame is a straightforward generalization of (\cite{Streetssolitons} Lemma 5.7), which we briefly sketch.  Fix $p \in M$ and a chosen local complex coordinates $z_j = x_j + \i y_j$ so that $\del_{x_n}(p) = V, \del_{y_n}(p) = JV$, and the vectors $\{\del_{x_1}, \del_{y_1}, \dots, \del_{x_{n-1}}, \del_{y_{n-1}}\}$ span $\mathcal H$ at $p$.  Now let
\begin{align*}
    e_{2i - 1} =&\ \del_{x_i} - \eta(\del_{x_i}) V - J \eta (\del_{x_i}) JV\\
    e_{2i} =&\ \del_{y_i} - \eta(\del_{y_i}) V - J \eta(\del_{y_i}) JV.
\end{align*}
These clearly span $\mathcal H$, and using the properties of $\mu$ and $F$ properties (2)-(4) follow easily.
\end{proof}

\begin{rmk} \label{r:ON} If we scale the metric $g$ so that $\brs{V} = \brs{JV} = 1$,   then furthermore in the construction of Lemma \ref{l:frame} we can assume that the vectors $\{e_1,\dots,e_{2n-2}\}$ are an orthonormal basis for the horizontal space at the chosen point $p$, so that the whole adapted frame is orthonormal at $p$.  
\end{rmk}

Next, to further analyze the torsion, let us compute the Levi-Civita covariant derivative of $H$. 

\begin{lemma}\label{L5}
Let $(M^{6},g,J,H)$ be a compact BHE manifold with trivial soliton potential.  The nonvanishing components of the Levi-Civita covariant derivative of $H$ are \begin{enumerate}
    \item $(\nabla_i H)_{\alpha kl}=(\nabla_iF_V)_{kl},$ $(\nabla_i H)_{\beta kl}=(\nabla_iF_{JV})_{kl}$.
    \item $(\nabla_iH)_{\alpha\beta j}= -\tfrac{1}{2}R^B_{\alpha\beta ij}$.
    \item $(\nabla_\alpha H)_{\beta ij}= - (\nabla_\beta H)_{\alpha ij}= \tfrac{1}{2}R^B_{\alpha\beta ij}$.
\end{enumerate}
\end{lemma}

\begin{proof}
By \Cref{Ch}, one immediately sees that the terms in items (1)-(3) are the only potentially non-trivial cases, except perhaps $(\N_i H)_{jkl}$.  However, using \Cref{P3}, we see
\begin{align*}
    (\nabla_iH)_{jkl}&=e_i(H_{jkl})-H(\nabla_{e_i}e_j,e_k,e_l)-H(e_j,\nabla_{e_i}e_k,e_l)-H(e_j,e_k,\nabla_{e_i}e_l)
    \\&=\tfrac{1}{2}\big( F_{ij\alpha}F_{kl\alpha}+F_{jk\alpha}F_{il\alpha}-F_{jl\alpha}F_{ik\alpha}  \big)+\tfrac{1}{2}\big( F_{ij\beta}F_{kl\beta}+F_{jk\beta}F_{il\beta}-F_{jl\beta}F_{ik\beta}  \big)
    \\&=\tfrac{1}{2}\big( (F_V\wedge F_V)_{ijkl}+(F_{JV}\wedge F_{JV})_{ijkl} \big),
    \\&= 0.
\end{align*}
Next we compute
\begin{align*}
    (\nabla_i H)_{\alpha kl}&=e_i\big( H_{\alpha kl}  \big)-H(\nabla_{e_i}e_{\alpha},e_k,e_l)-H(e_{\alpha},\nabla_{e_i}e_k,e_l)-H(e_{\alpha},e_k,\nabla_{e_i}e_l)
    \\&=e_i\big( (F_V)_{kl} \big)-\Gamma^g_{ikm}(F_V)_{ml}-\Gamma^g_{ilm}(F_V)_{km}
    \\&= (\nabla_iF_V)_{kl},
\end{align*}
\begin{align*}
     (\nabla_iH)_{\alpha\beta j}&=e_i(H_{\alpha\beta j})-H(\nabla_{e_i}e_\alpha,e_\beta,e_j)-H(e_\alpha,\nabla_{e_i}e_\beta,e_j)-H(e_\alpha,e_\beta,\nabla_{e_i}e_j)
     \\&=-\tfrac{1}{2}F_{il\alpha}F_{jl\beta}-\tfrac{1}{2}F_{il\beta}F_{lj\alpha}
     \\&=-\tfrac{1}{2}R^B_{\alpha\beta ij},
\end{align*}
\begin{align*}
    (\nabla_\alpha H)_{\beta ij}&=e_{\alpha}(H_{\beta ij})-H(\nabla_{e_\alpha}e_\beta,e_i,e_j)-H(e_\beta,\nabla_{e_\alpha}e_i,e_j)-H(e_\beta,e_i,\nabla_{e_\alpha}e_j)
     \\&=\tfrac{1}{2}F_{il\alpha}F_{jl\beta}-\tfrac{1}{2}F_{jl\alpha}F_{il\beta}
     \\&=\tfrac{1}{2}R^B_{\alpha\beta ij}.
\end{align*}
\end{proof}

Next, we compute the components of Riemannian curvature.
\begin{lemma}\label{L6}
Let $(M^{6},g,J,H)$ be a compact BHE manifold with trivial soliton potential.  The Riemann curvature satisfies 
\begin{enumerate}
    \item $R_{\alpha ijk}=\tfrac{1}{2}R^B_{\alpha ijk}=-\tfrac{1}{2}(\nabla_i F_V)_{jk}$, $R_{\beta ijk}=\tfrac{1}{2}R^B_{\beta ijk}=-\tfrac{1}{2}(\nabla_i F_{JV})_{jk}$.
    \item $R_{\alpha\beta ij}=\tfrac{1}{4}R^B_{\alpha\beta ij}$.
    \item $R_{\alpha ij\beta}=\tfrac{1}{4}F_{jk\alpha}F_{ik\beta}$.
\end{enumerate}

\end{lemma}

\begin{proof}
  From the definition of Bismut curvature, we have
  \begin{align*}
      R^B(X,Y,Z,W)=& \kern0.2em R(X,Y,Z,W)+\tfrac{1}{2}(\nabla_XH)(Y,Z,W)-\tfrac{1}{2}(\nabla_YH)(X,Z,W)
    \\&-\tfrac{1}{4}\langle H(X,W),H(Y,Z)\rangle+\tfrac{1}{4}\langle H(Y,W),H(X,Z)\rangle. 
  \end{align*}
Then, we compute 
\begin{align*}
    R^B_{\alpha ijk}&=R_{\alpha ijk}+\tfrac{1}{2}(\nabla_{\alpha}H)_{ijk}-\tfrac{1}{2}(\nabla_iH)_{\alpha jk}-\tfrac{1}{4}H_{\alpha k E}H_{ij E}+\tfrac{1}{4}H_{\alpha jE}H_{ik E}
    \\&=R_{\alpha ijk}+\tfrac{1}{2}R^B_{\alpha ijk},
\end{align*}
\begin{align*}
    R^B_{\alpha\beta ij}&=R_{\alpha\beta ij}+\tfrac{1}{2}(\nabla_{\alpha}H)_{\beta ij}-\tfrac{1}{2}(\nabla_\beta H)_{\alpha ij}-\tfrac{1}{4}H_{\alpha j E}H_{\beta i E}+\tfrac{1}{4}H_{\alpha iE}H_{\beta j E}
    \\&=R_{\alpha\beta ij}+\tfrac{1}{2} R^B_{\alpha\beta ij}+\tfrac{1}{4} R^B_{\alpha\beta ij},
\end{align*}  
\begin{align*}
    0 = R^B_{\alpha ij\beta}&=R_{\alpha ij\beta}+\tfrac{1}{2}(\nabla_{\alpha}H)_{ij\beta}-\tfrac{1}{2}(\nabla_i H)_{\alpha j\beta}-\tfrac{1}{4}H_{\alpha \beta E}H_{ij E}+\tfrac{1}{4}H_{\alpha jE}H_{i\beta  E}
    \\&=R_{\alpha ij\beta}-\tfrac{1}{4}F_{jk\alpha}F_{ik\beta},
\end{align*}
where we used results from \Cref{L4} and \Cref{L5}.
\end{proof}

Using \Cref{L6} we derive a further vanishing result.
\begin{lemma}\label{L7}
   Let $(M^{6},g,J,H)$ be a  non-K\"ahler BHE manifold with trivial soliton potential and suppose that the harmonic anti-self-dual   $2$-forms $\alpha^T$  and  $\gamma^T:= \frac{\Scal(\omega_K)}{4} \omega^T_K- \rho_K^T$ in Proposition~\ref{p:4D} are linearly dependent at each point.  Then
   \begin{align*}
       R_{\alpha\beta ij}=0.
   \end{align*}
\end{lemma}

\begin{proof} Assuming $f=const$, we know that $\omega^T=\frac{\Scal(\omega_K)}{2}g_K$ is K\"ahler,  and then  \[ F_{JV}=-\rho_K^T =-\tfrac{1}{2} \omega_T + \gamma^T,\]  where $\gamma^T$ is a (harmonic) anti-self-dual $2$-form.  Our assumption is that $\alpha^T$ and $\gamma^T$ are linearly dependent at any point. By Lemma~\ref{L6}(2) and Lemma~\ref{L4}(2) we have
\[R(V, JV) = g^T[F_V\, (g^T)^{-1}, F_{JV} \,  (g^T)^{-1}]= g^T[\alpha^T \,  (g^T)^{-1}, \gamma^T \, (g^T)^{-1}] =0,\]
as required.
\end{proof}

\begin{lemma} \label{LP4}
 Under the assumptions of Lemma~\ref{L7},   $|F_V|$ and $|F_{JV}|$ are both constant.
\end{lemma}

\begin{proof}
  We have by Proposition~\ref{P3}
\begin{align}\label{eq:a}
    F_V=\alpha^T \qquad F_{JV}=- \tfrac{1}{2}\omega^T+\gamma^T,
\end{align}
where $\alpha^T, \gamma^T$ are both harmonic and anti-self-dual $2$-forms on $(M^T, g^T)$, using that $f$ is constant.  Note that $\alpha^T$ and $\gamma^T$ cannot be both identically zero as by Proposition~\ref{P3}-(3) we have \begin{equation}\label{eq:b}
|\gamma^T|^2 + |\alpha^T|^2  \equiv  1.
\end{equation} 
The proof now breaks into two cases:
\vskip 0.1in

\noindent \textbf{Case 1:} $\gamma^T\equiv 0$ around a point.  In this case \eqref{eq:a} and \eqref{eq:b} yield 
\begin{align*}
    |F_V|^2 = |\alpha^T|^2 = 1 = 1+ |\gamma^T|^2=|F_{JV}|^2,
\end{align*}
so the statement holds around that point.

\vskip 0.1in

\noindent \textbf{Case 2:} $\gg^T \not \equiv   0$ around a point.  Working on a local leaf space $M^T$ around such a point, and using that a non-trivial harmonic form cannot be zero on an open subset,
we know that $\gamma^T \neq 0$ on a dense open subset $U$ of $M^T$. Over $U$, we can write 
$\alpha^T = \lambda \gamma^T$ for a smooth function $\lambda$. As $\alpha^T$ and $\gamma^T$ are both closed, we have $d\lambda \wedge \gamma^T=0$.
As an anti-self-dual $2$-form satisfies $\gamma^T \wedge \gamma^T = -\frac{|\gamma^T|^2}{2}dV_{g^T}$, we know that $\gamma^T$ is non-degenerate over $U$, and thus $\lambda=const$ on $U$, i.e. $\alpha^T = \lambda \gamma^T$ on $U$ with $\lambda$ constant. This relation extends by continuity to $M^T$ as $U$ is dense.  It now follows  by \eqref{eq:a} and \eqref{eq:b} that
\begin{align*}
    1= &|\alpha^T|^2+|\gamma^T|^2=(1+\lambda^2)|\gamma^T|^2.
\end{align*}
Thus, $|\gamma^T|^2$  and $|\alpha^T|^2 = \lambda^2 |\gamma^T|^2$ are constant, and therefore  $\brs{F_V}^2=|\alpha^T|^2$ and $\brs{F_{JV}}^2= 1 + |\gamma^T|^2$ are constant too. \end{proof}

\subsection{Parallel torsion}

We give here a general proposition about the Bismut covariant derivative of $H$ applying in all dimensions.

\begin{prop}\label{prop:cov deriv} 
    Let $(M^{2n},J,g)$ be a compact BHE manifold.  If $H$ is Bismut parallel then $f$ is constant and $F_{V}$ and $F_{JV}$ are parallel with respect to the Bismut connection.  In the case $n=3$, if $F_V$ and $F_{JV}$ are parallel with respect to the transverse Levi-Civita connection and $f$ is constant, then $H$ is Bismut parallel.
\end{prop}
\begin{proof}
    Since $V$ and $JV$ are $\nabla^B$-parallel, taking the covariant derivative of \eqref{eq: torsion form} gives
    \begin{align} \label{f:covH}
    \nabla^B H = \nabla^B H^T + (\nabla^B F_V)\wedge\eta + (\nabla^B F_{JV})\wedge J\eta.
    \end{align}
    Notice by \Cref{Ch} that the covariant derivative $\nabla^B \sigma$ of a horizontal form $\sigma$ is horizontal.  Therefore, $H$ is $\nabla^B$-parallel if and only if $H^T$, $F_V$ and $F_{JV}$ are so.
    We recall that $H^T=-Jd\omega^T$. Thus, by wedging $\nabla^B H^T$ with $\left(\omega^T\right)^{n-3}$ we get that $\nabla^B\theta_{g^T}=0$. Then, \Cref{P3} leads to $\nabla^Bdf=0$ which implies $f$ constant since $M$ is compact.
    In complex dimension $3$, assuming that $f$ is constant, the transverse distribution is K\"ahler by \Cref{P3}.  In particular $H^T = 0$ and on the horizontal distribution $\N^B = \N$, hence the claim follows from (\ref{f:covH}).
\end{proof}

\subsection{Proof of Corollary~\ref{c:rigiditycor}}

\begin{proof}[Proof of Corollary \ref{c:rigiditycor}] We may assume without loss of generality that $(M, \omega, J)$ is non-K\"ahler, so the ansatz of \S \ref{s:dimred} holds.  As we have explained at the beginning of the section, each of the cases (1)-(3) of Corollary~\ref{c:rigiditycor} reduces to consider transversal K\"ahler structure $(M^T, J^T, \omega_K^T)$ of constant positive scalar curvature
$\Scal(\omega_K)$ and such that the transversal harmonic anti-self-dual $2$-forms $F_V=\alpha^T$ and $(F_{JV})_-=\gamma^T$ are proportional. 
Notice that, as $\rho_K^{T}= -F_{JV}$, the traceless Ricci tensor is 
\begin{align} \label{f:mainpf10}
    \mathring{\Rc}^{g^T}= -(\rho_K^T \circ J)^{\mathring{}}= \gamma^T \circ J. 
\end{align}
From Lemma~\ref{LP4}, we deduce that $\brs{\mathring{\Rc^{g^T}}}^2$ is constant so $\Rc^{g^T}$ has constant eigenvalues $\lambda \le \mu$, each of multiplicity $2$;  furthermore, by Proposition \ref{p:EM}, we have  $0 \le \lambda \le \mu$. This is a situation studied in \cite{VTA}.

We claim that $F_{J V}$ is parallel.  We split the argument into two cases, according to whether the eigenvalues $\lambda, \mu$ are distinct or not.  First, if the eigenvalues are equal, then $\mathring{\Rc}^{g^T} = 0$, thus $(F_{J V})_- = 0$, and $F_{JV}$ is parallel.  In case the eigenvalues are distinct, we adapt the argument of (\cite{VTA} Theorem 1.1) as follows.  The eigenspaces for the transverse Ricci tensor define an orthogonal splitting of the horizontal tangent space ${\mathcal H}= E_\lambda \oplus E_\mu$ at each point.  By  \cite{VTA} Lemma 2.1,  we obtain a second transverse almost complex structure $\bar{J}$, which defines a transverse almost K\"ahler structure $(g^T, \bar J)$ on $\mathcal H$.  Furthermore, as explained in \cite{VTA} Lemma~2.1, restricting the K\"ahler form $\omega$ to these two sub-bundles yields  $(1,1)$-forms $\ga, \gb$, which are closed.  Next we apply the argument of (\cite{VTA} Theorem 1.1) to show that the K\"ahler form $\bar{\omega}= g^T \bar J$ associated to $(g^T,\bar{J})$ is parallel with respect to the transverse Levi-Civita connection.  Observe that this argument depends on (\cite{VTA} Proposition 2.1),  which in turn is a pointwise statement which therefore holds in our situation on any local quotient of $M$ by the flows of $Z, W$.  Furthermore, the integral identity in  (\cite{VTA} Corollary 2.2) also extends in our situation as for any transversal forms $\gamma^T, \delta^T$, the $L^2$ inner product on $(M, g)$  satisfies   $\langle \pi^* d^*_{g^T} \gg^T, \pi^* \delta^T \rangle = \langle d^*_g (\pi^* \gg^T), \pi^* \delta^T \rangle$. Thus it follows that $\bar{\omega}$ is transversally parallel, and thus the transverse Ricci tensor is parallel.  Then finally $F_{J V}$ is parallel by (\ref{f:mainpf10}).  Turning to $F_{V}$, since it is trace-free, of type $(1,1)$ and of constant norm, it either vanishes (and is thus obviously parallel), or defines a nontrivial splitting ${\mathcal H}= E_{\lambda} \oplus E_{\mu}$ of the horizontal space as before.  By the argument above, this corresponds to a second K\"ahler structure $(g, \bar{J})$, with parallel K\"ahler form $\bar{\gw} = \gl F_{V}$ for a nonzero constant $\gl$.  Thus it follows that $F_{V}$ is parallel.

Thus we have shown that $F$ is parallel, and then it follows from \Cref{prop:cov deriv} that $\N^B H = 0$.  It follows from (\cite{brienza2024cyt} Theorem 1.1) (cf. also (\cite{barbaro2024pluriclosed} Theorem B)) that the universal cover is a product of a K\"ahler-Ricci flat space and a simply connected Bismut-flat space. Aside from the case $(M, g, J)$ is already K\"ahler Calabi-Yau, the K\"ahler factor can have dimension at most $1$, and is hence flat, thus the total space is Bismut flat.  The result then follows from (\cite{ZhengBismutflat} Corollary 4), noting that the BHE geometries of $\SU(2) \times \mathbb R \times \mathbb C$ correspond to the case when the transversal Ricci tensor ${\rm Rc}^{g^T}$  has constant eigenvalues  $0=\lambda<\mu$ (and hence $\alpha^T=0$) whereas the BHE geometries of $\SU(2)\times \SU(2)$ corresponds to the case when the transversal Ricci tensor ${\rm Rc}^{g^T}$  has constant eigenvalues $0<\lambda \le \mu$.
\end{proof}

\textbf{Conflict of interest statement:} On behalf of all authors, the corresponding author states that there is no conflict of interest.

\vspace{15pt}

\textbf{Data availability statement:} Data sharing is not applicable to this article as no datasets were generated or analyzed during the current study.


\end{document}